\documentclass{article}
\usepackage{amsthm, amssymb, amsmath, mathrsfs, dsfont, color, graphicx, dutchcal,hyperref}
\newtheorem{theorem}{Theorem}
\newtheorem{lemma}{Lemma}
\newtheorem{corollary}{Corollary}
\newtheorem{proposition}{Proposition}
\newtheorem{claim}{Claim}

\def \E {{\mathbb E}}
\def \L {{\mathbb L}}
\def \N {{\mathbb N}}
\def \P {{\mathbb P}}

\def \Z {{\mathbb Z}}
\def \T {{\mathbb T}}

\newcommand{\cals}{\mathbcal{s}}
\newcommand{\CP}[4][]{\if\relax\detokenize{#1}\relax \xi_{#2,#3;#4}^{#2} \else \xi_{#2,#3;#4}^{#1} \fi}
\newcommand{\barCP}[3][]{\if\relax\detokenize{#1}\relax \bar{\xi}_{#2,#3}^{#2} \else \bar{\xi}_{#2,#3}^{#1} \fi}
\newcommand{\tildeCP}[5][]{\if\relax\detokenize{#1}\relax \tilde{\Xi}_{#2,#3}^{#2}(#4,#5) \else \tilde{\Xi}_{#2,#3}^{#1}(#4,#5) \fi}

\newcommand{\specificthanks}[1]{\@fnsymbol{#1}}

\begin{document}
	
	\title{Graph constructions for the contact process with a prescribed critical rate}
	
	\renewcommand{\thefootnote}{\arabic{footnote}}
	\author{
		\renewcommand{\thefootnote}{\arabic{footnote}} Stein Andreas Bethuelsen\footnotemark[1], \\
		Gabriel Baptista da Silva\footnotemark[1] \textsuperscript{,2}, \\
		\renewcommand{\thefootnote}{\arabic{footnote}} Daniel Valesin\footnotemark[2]}
	
	\renewcommand{\thefootnote}{\fnsymbol{footnote}}
	\footnotetext[1]{Corresponding author. Email: \url{g.da.silva@rug.nl}}
	\renewcommand{\thefootnote}{\arabic{footnote}}
	\footnotetext[1]{Department of Mathematics, University of Bergen, Norway}
	\footnotetext[2]{Bernoulli Institute, University of Groningen, The Netherlands}
	\date{\today}
	\maketitle
	
	\begin{abstract}
		We construct graphs (trees of bounded degree) on which the contact process has critical rate (which will be the same for both global and local survival) equal to any prescribed value between zero and~$\lambda_c(\mathbb{Z})$, the critical rate of the one-dimensional contact process.	
		We exhibit both graphs in which the process at this target critical value survives (locally) and graphs where it dies out (globally).
	\end{abstract}

	\section{Introduction}
	\label{sec:intro}
	
	This paper exhibits a range of examples concerning phase transitions of the contact process. Our work can be seen as a complement to the previous works by Madras, Schinazi and Schonmann~\cite{mss}, and by Salzano and Schonmann~\cite{ss1,ss2}, where the same line of inquiry was pursued.

	The contact process describes a class of interacting particle systems which serve as a model for the spread of epidemics on a graph. It was introduced by Harris~\cite{harris}. It is defined on a graph~$G$ with uniformly bounded degrees by the following rules for a continuous-time Markov dynamics: vertices can be healthy (state 0) or infected (state~1); infected vertices recover with rate one, and transmit the infection to each healthy neighbour with rate~$\lambda > 0$. The above description of the dynamics expresses (on graphs of uniformly bounded degree) a Markov pre-generator on a dense subspace of the space of real-valued functions on the space of configurations, endowed with the topology of the supremum norm. The closure of this pre-generator is a Feller generator, making the contact process a Feller process. See Chapter 1 of~\cite{lig1} for more details.
	
	We denote by~$(\xi^A_{G,\lambda;t}: t \ge 0)$ the contact process on~$G = (V,E)$ with infection rate~$\lambda$ and initially infected set~$A \subset V$ (as explained in Section~\ref{ss:graphical}, we will occasionally omit or change aspects of this notation). With a conventional abuse of notation, we treat~$\xi^A_{G,\lambda;t}$ as either an element of~$\{0,1\}^V$ or as a subset of~$V$ (the set of infected vertices). We refer the reader to~\cite{lig1} and~\cite{lig2} for an introduction to this process, including all the statements made without further explicit reference in this introduction.

	The contact process has as absorbing state the configuration in which all individuals are healthy; we denote this state by~$\varnothing$. We define the probability of survival (of the infection)
	$$\zeta_{G,\lambda}(A) := \mathbb{P}\left(\xi^A_{G,\lambda;t} \neq \varnothing \text{ for all }t\right),\quad A \subset V.$$
	Due to an elementary monotonicity property of the process, this quantity is non-decreasing in~$\lambda$,~$G$ and~$A$ (for the latter two, take the partial order given by graph and set inclusion, respectively). Moreover, if~$G$ is connected, then for any~$\lambda$,~$\zeta_{G,\lambda}(A)$ is either equal to zero for all finite~$A$ (in which case the process with parameter~$\lambda$ on~$G$ is said to die out) or non-zero for any finite and non-empty~$A$ (the process is then said to survive, or to survive globally).
	We then define the critical threshold for global survival as
	$$\lambda^\mathrm{glob}_{c}(G) := \inf\left\{\lambda: \zeta_{G,\lambda}(A) > 0 \text{ for all (any) finite and non-empty }A \right\}.$$
	Next, define the probability of local survival
	$$\beta_{G,\lambda}(A,v):= \mathbb{P}\left(\limsup_{t \to \infty}\xi^A_{G,\lambda;t}(v) = 1\right),\quad A \subset V,\; v \in V.$$
	It is readily seen that~$\beta_{G,\lambda}(A,v) \leq \zeta_{G,\lambda}(A)$. Moreover,~$\beta_{G,\lambda}(A,v)$ is non-decreasing  in~$\lambda,G,A$, and if~$G$ is connected, then for fixed~$\lambda$ we either have $\beta_{G,\lambda}(A,v) = 0$ for all choices of (finite, non-empty)~$A$ and~$v$, or~$\beta_{G,\lambda}(A,v) > 0$ for all such choices. In the latter case, we say that the process survives locally (in other sources, it is said in this case that the process survives strongly, or is recurrent). We define the critical threshold for local survival as
	$$\lambda^\mathrm{loc}_c(G) := \inf\left\{\lambda: \beta_{G,\lambda}(A,v) > 0 \text{ for all(any)~$v$ and finite }A \neq \varnothing \right\}.$$
	
	The contact process has been initially studied on~$\mathbb{Z}^d$; there it holds that the two critical values coincide; we will denote their common value by~$\lambda_c(\mathbb{Z}^d)$. It was proved in~\cite{bezui} that the process on~$\mathbb{Z}^d$ at the critical rate dies out. Results for the contact process on the infinite regular tree with offspring number~$d \ge 2$ (denoted~$\mathbb{T}^d$) were obtained in the~1990's, notably in~\cite{lig_tree_a},~\cite{lig_tree_b} and~\cite{pem_tree}. There it holds that~$0<\lambda_c^\mathrm{glob}(\mathbb{T}^d) <\lambda_c^\mathrm{loc}(\mathbb{T}^d) < \infty$, and moreover the  process at the lower critical value dies out, and the process at the upper critical value survives globally but not locally. More recently, results for the contact process on random trees have gained interest. In particular, \cite{hd} and \cite{bnns} completely characterize the existence of a subcritical regime for the process on Galton-Watson trees (see also \cite{nns}).
	
	The main result of this paper concerns the set of values that the critical rates~$\lambda_c^\text{glob}(G),\;\lambda_c^\text{loc}(G)$ can attain, as~$G$ ranges over any locally finite graph, and also whether the critical contact process can survive for these possible values of the critical rate. Let us make some preliminary comments in this direction.
	
	\begin{enumerate}
\item On a finite graph~$G$, the contact process dies out regardless of~$\lambda$, that is, we have~$\lambda^\mathrm{glob}_c(G) = \lambda^\mathrm{loc}_c(G) = \infty$. 
\item On an infinite graph~$G$, we necessarily have~$\lambda_c^\mathrm{glob}(G) \le \lambda_c^\mathrm{loc}(G) \le \lambda_c(\mathbb{Z})$. This follows from monotonicity:~$G$ contains a copy of~$\mathbb{N}$ inside it (since~$G$ is locally finite), and it is known that~$\lambda_c(\mathbb{N}) =\lambda_c(\mathbb{Z})$; see for instance Corollary~2.5 in~\cite{amp}. 
\item There are infinite graphs for which the critical rate for local (hence also global) survival is arbitrarily small, such as high-dimensional lattices and high-degree regular trees, see~\cite[equation~(1.14)]{griff} and~\cite[Theorem~2.2]{pem_tree}
\item There are also infinite graphs for which the critical rate for local (hence also global) survival is equal to zero, such as Galton-Watson trees with sufficiently heavy-tailed offspring distributions, see~\cite[page~2112]{pem_tree}.
\item  An example was given in~\cite{ss2} of a graph~$G$ with~$\lambda^\mathrm{loc}_c(G) = \lambda^\mathrm{glob}_c(G) = \lambda_c(\mathbb{Z})$ and so that the contact process with this critical rate survives locally. This is the ``desert-and-oasis'' example in page 863 of that paper, which is based on a construction of~\cite{mss} pertaining to a contact process with inhomogeneous rates.
\item In pages~859-862 of~\cite{ss2}, the authors fix~$d \ge 2$, then fix an arbitrary~$\lambda$ with~$\lambda_c^\mathrm{glob}(\mathbb{T}^d )< \lambda < \lambda_c^\mathrm{loc}(\mathbb{T}^d)$, and construct a graph~$G$ for which~$\lambda = \lambda_c^\mathrm{glob}(G) < \lambda_c^\mathrm{loc}(G)$. The class of examples obtained in this way therefore shows that
\begin{equation}\forall \lambda \in \bigcup_{d=2}^\infty (\lambda^\mathrm{glob}_c(\mathbb{T}^d),\lambda^\mathrm{loc}_c(\mathbb{T}^d))\quad \exists G: \lambda = \lambda_c^\mathrm{glob}(G) < \lambda_c^\mathrm{loc}(G). \label{eq:gaps}\end{equation}
\end{enumerate}
	
	We now state our main result:	
	\begin{theorem}\label{thm:main}\begin{itemize}
			\item[(a)] For any~$\lambda \in (0,\lambda_c(\mathbb{Z}))$ there exists a tree~$G$ of bounded degree for which $\lambda_c^\mathrm{glob}(G) = \lambda_c^\mathrm{loc}(G)= \lambda$ and the contact process on~$G$ with rate~$\lambda$ survives locally.
			\item[(b)] For any~$\lambda \in (0,\lambda_c(\mathbb{Z}))$ there exists a tree~$G$ of bounded degree such that $
			\lambda_c^\mathrm{glob}(G) = \lambda_c^\mathrm{loc}(G)= \lambda$ and the contact process on~$G$ with rate~$\lambda$ dies out.
		\end{itemize}
	\end{theorem}
	Together with \cite{ss2}, the above theorem provides a full answer to the question of which values $\lambda > 0$ can occur simultaneously as $\lambda_c^\mathrm{loc}(G)$ and $\lambda_c^\mathrm{glob}(G)$ for some locally finite connected graph $G$. Indeed the case $\lambda = \lambda_c(\mathbb{Z})$ was covered in \cite{ss2} and, as noted in Comment 2.\  above,  for every infinite graph $G$ one has~$\lambda_c(G) \leq \lambda_c(\mathbb{N}) = \lambda_c(\mathbb{Z})$.%due to the monotonicity of the contact process and the fact that one can always find, on an infinite connected graph in which all degrees are finite, a subgraph that is isomorphic to $\mathbb{N}$.}	
	 
	Although the construction we give here is very similar to the one in~\cite{ss2} (and~\cite{mss}) mentioned above, it has novel aspects that free us from being restricted to having~$\lambda_c(\mathbb{Z})$ as the critical rate. In essence, the graph we construct consists of an infinite half-line to which we append, in very sparse locations (say,~$a_1 \ll \cdots \ll a_i \ll \cdots$), regular trees with large (but fixed) degree, truncated at height~$h_i$. In the terms of the aforementioned examples of~\cite{mss} and~\cite{ss2}, the half-line is the ``desert'' and the trees are the ``oases''. This means that, for~$\lambda$ within a certain controlled range (inside the interval~$(0,\lambda_c(\mathbb{Z}))$), the contact process stays active for a very long time in the trees, but is very unlikely to cross the line segments in between them in any single attempt. The locations and heights are chosen in a way that is increasingly sensitive to the value of~$\lambda$, so that a certain target value can be guaranteed to be critical for global and local survival.
	
We should mention that, in case one does not insist in obtaining graphs of bounded degree in the statement of Theorem~\ref{thm:main}, then the oasis structures could be taken as stars of increasing degree instead of trees of increasing height~\footnote{Some care must be taken before one considers the contact process on graphs with unbounded degrees. The standard construction, as described in \cite{lig1} and \cite{lig2}, is valid on graphs with bounded degree; in the absence of this condition, the assumptions required to define a Feller semi-group from a pre-generator may fail to hold. However, for the alternate graph outlined here (which alternates line segments of increasing length with stars of increasing degree), and the contact process started from configurations with finitely many infected vertices, there are no problems with constructing the process. Indeed, in that case it can be constructed more simply as a continuous-time Markov chain on a countable state space (namely, the collection of finite sets of vertices), and the structure of the graph makes it easy to see that explosion does not occur (by considering the time it takes for the infection to cross each line segment). Since the alternate graph is not the focus of this work, we omit further details.}. Taking stars rather than trees would indeed simplify some of our proofs somewhat. Moreover, even keeping degrees bounded, other structures would also work instead of trees, such as high-dimensional hypercubes. We have chosen to use trees because some estimates and coupling results were readily available for the contact  process on trees in the reference~\cite{cmmv}.

\subsection{Open questions}

Let us first mention that we believe the ideas we develop in this paper allow for graph constructions that lead to replacing the union in~\eqref{eq:gaps} by the full interval~$(0,\lambda_c(\mathbb{Z}))$, but we do not work out the details here. 

We will now briefly discuss questions that we consider interesting and that would be further developments to our result.

 	\textbf{Question 1.} Can one construct a locally finite connected graph $G$ for which~$\lambda_c^\mathrm{glob}(G) = \lambda_c^\mathrm{loc}(G)$ and the contact process on $G$ dies out locally but survives globally?
 
 	\textbf{Question 2.} What is the set of pairs $(\lambda_1,\lambda_2) \in [0,\lambda_c(\mathbb{Z})]^2$ that can occur as~$(\lambda_c^\mathrm{glob}(G),\lambda_c^\mathrm{loc}(G))$ for some graph~$G$?
	
	\textbf{Question 3.} Fix any (finite or infinite) sequence of values~$0 < \lambda_1 < \lambda_2 < \cdots < \lambda_c(\mathbb{Z})$. Is there a graph~$G$ for which the function~$\lambda \mapsto \zeta_{G,\lambda}(A)$ (for any~$A$) is discontinuous at~$\lambda_i$ for each~$i$? It is conceivable that, by glueing together graphs obtained from Theorem~\ref{thm:main}, each with a different critical value, one would find and affirmative answer to this question. See the proof of Theorem~3.2.1 in \cite{ss2} for an instance where glueing graphs can produce this kind of discontinuity.

%	The rest of the paper is organized as follows. In the rest of this introduction, we explain the notation we use and the graphical construction of the contact process. In Section~\ref{s:proof_of_main}, we state Theorem~\ref{thm:tool}, which allows us to augment  graphs in a way that is favorable for the contact process with rate~$\lambda$ and unfavorable for the process with rate~$\lambda' < \lambda$, where~$\lambda$ is some prescribed infection rate. Using this theorem, we give in that section the proof of Theorem~\ref{thm:main}; the remainder of the paper is dedicated to the proof of Theorem~\ref{thm:tool}. Section~\ref{s:line_tree} gathers some preliminary results about the contact process on line segments and trees. Section~\ref{sec:proof_thm} contains the key definitions of our graph augmentation construction, and states key results (Propositions~\ref{prop:ignition},~\ref{prop:time_spent} and~\ref{prop:no_go}), which together readily give the proof of Theorem~\ref{sec:proof_thm}. Section~\ref{s:proofs_of_props} and the appendix are more technical and contain the proofs of the three key propositions (as well as several auxiliary results).

\subsection{Organization of the paper}

In the rest of this introduction, we explain the notation we use and the graphical construction of the contact process. In Section~\ref{s:proof_of_main}, we state Theorem~\ref{thm:tool}, which allows us to augment  graphs in a way that is favorable for the contact process with rate~$\lambda$ and unfavorable for the process with rate~$\lambda' < \lambda$, where~$\lambda$ is some prescribed infection rate. Using this theorem, we give in that section the proof of Theorem~\ref{thm:main}; the remainder of the paper is dedicated to the proof of Theorem~\ref{thm:tool}. Section~\ref{s:line_tree} gathers some preliminary results about the contact process on line segments and trees. Section~\ref{sec:proof_thm} contains the key definitions of our graph augmentation construction, and states key results (Propositions~\ref{prop:ignition},~\ref{prop:time_spent} and~\ref{prop:no_go}), which together readily give the proof of Theorem~\ref{sec:proof_thm}. Section~\ref{s:proofs_of_props} and the appendix are more technical and contain the proofs of the three key propositions (as well as several auxiliary results).

\subsection{Notation, graphical construction}
\label{ss:graphical}

Let us first detail the notation we use for graphs. Let~$G = (V,E)$ be an unoriented graph with set of vertices~$V$ and set of edges~$E$. We say two vertices are neighbors if there is an edge containing both. The degree of a vertex~$v$, denotes~$\mathrm{deg}_G(v)$, is the number of neighbors of~$v$. All graphs we consider are locally finite, meaning that all their vertices have finite degree. Finally, graph distance in~$G$ between vertices~$u$ and~$v$ is denoted~$\mathrm{dist}_G(u,v)$.

Next, we recall the graphical construction of the contact process. Here we consider a standard monotone coupling of contact processes on the same graph with different infection rates. This is implemented by endowing transmission arrows with numerical labels, as we now explain. Fix a graph~$G$ and also~$\lambda > 0$. We take a family of independent Poisson point processes:
\begin{itemize}
	\item for each~$v \in V$, a Poisson point process~$D^v$ on~$[0,\infty)$ with intensity equal to Lebesgue measure; if~$t \in D^{u}$, we say there is a recovery mark at~$u$ at time~$t$;
	\item for each ordered pair~$(u,v) \in V^2$ such that~$\{u,v\} \in E$, a Poisson process~$D^{(u,v)}$ on~$[0,\infty)^2$ with intensity equal to Lebesgue measure; if~$(t,\ell) \in D^{(u,v)}$, we say there is a transmission arrow with label~$\ell$ at time~$t$ from~$u$ to~$v$.
	\end{itemize}
Given~$\lambda > 0$ and~$u,v \in V$ and~$0 \le s < t$, a~$\lambda$-infection path from~$(u,s)$ to~$(v,t)$ is a right-continuous function~$\gamma: [s,t] \to V$ satisfying~$\gamma(s) =u$,~$\gamma(t) = v$,
\begin{align*} & r \notin D^{\gamma(r)}  \text{ for all } r,\text{ and}\\[.2cm] & \text{whenever } \gamma(r-) \neq \gamma(r) \text{ there is } \ell \le \lambda \text{ such that } (r,\ell) \in D^{(\gamma(r-),\gamma(r))}.\end{align*}
That is, a~$\lambda$-infection path cannot touch recovery marks and can traverse transmission arrows with label smaller than or equal to~$\lambda$.

In most places, the value of~$\lambda$ will be clear from the context, so we simply speak of infection paths rather than~$\lambda$-infection paths. We write~$(u,s) \stackrel{\lambda}{\rightsquigarrow} (v,t)$ (sometimes omitting~$\lambda$) either if~$(u,s) = (v,t)$ or if there is a~$\lambda$-infection path from~$(u,s)$ to~$(v,t)$. More generally, for~$S_1,S_2 \subset V \times [0,\infty)$, we write~$S_1 \rightsquigarrow S_2$ if there is an infection path from~$(u,s)$ to~$(v,t)$, for some~$(u,s) \in S_1,~(v,t) \in S_2$ (we write~$S \rightsquigarrow (v,t)$ instead of~$S \rightsquigarrow \{(v,t)\}$, and similarly for~$(u,s) \rightsquigarrow S$). Given~$A \subset V$, setting
$$\CP[A]{G}{\lambda}{t}(v):= \mathds{1}\{A \times \{0\} \stackrel{\lambda}{\rightsquigarrow} (v,t)\},\quad t \ge 0,\; v \in V, $$
where~$\mathds{1}$ denotes the indicator function, we obtain that~$\CP[A]{G}{\lambda}{t}$ is a contact process with parameter~$\lambda$, started with vertices in~$A$ infected and vertices in~$V \backslash A$ healthy. Note that this construction readily gives the monotone relation
$$A \subset A',~G \text{ subgraph of }G',~\lambda  \le \lambda' \quad \Longrightarrow \quad \CP[A]{G}{\lambda}{t} \leq \CP[A']{G'}{\lambda'}{t},\; t \ge 0.$$
In case we are considering the contact process~$(\CP[A]{G}{\lambda}{t}:t \ge 0)$ on a graph~$G$ and~$G'$ is a subgraph of~$G$, we sometimes refer to~$(\CP[A]{G'}{\lambda}{t}:t \ge 0)$ as the process confined to~$G'$.

Finally, we write
$$\barCP[A]{G}{\lambda}(v):= \int_0^\infty \CP[A]{G}{\lambda}{t}(v)\;\mathrm{d}t,\quad v \in V,$$
that is,~$\barCP[A]{G}{\lambda}(v)$ is the total amount of time that~$v$ is infected in~$\left(\CP[A]{G}{\lambda}{t}:t \ge 0\right)$.

\section{Proof of main result}\label{s:proof_of_main}
Our graph construction will be given by recursively applying a graph augmentation procedure, with each step taking as input a rooted graph (a tree with bounded degree) and a prescribed value of the infection rate. The result that allows us to take each step is the following.
\begin{theorem}\label{thm:tool}
For any~$\lambda\in (0,\lambda_c(\mathbb{Z}))$ there exist~$c_\lambda > 0$ and~$d =d_\lambda \in \mathbb{N}$ satisfying the following. Let~$(G,o) = ((V,E),o)$ be a rooted tree with degrees bounded by~$d+1$, and~$\deg_G(o) = 1$. Then, there exists~$\mathcal{H} = \mathcal{H}((G,o),\lambda) \in \mathbb{N}$ such that for any~$h \ge \mathcal{H}$, there exists a rooted tree~$(\tilde{G},\tilde{o}) = (\tilde{G}_h,\tilde{o}_h)$ with vertices~$\tilde{V}$ and edges $\tilde{E}$ having~$G$ as a subgraph, with degrees satisfying
	\begin{align*}
	&\deg_{\tilde{G}}(v) = \deg_{G}(v)  \text{ for all } v \in V\backslash\{o\},\\
	&\deg_{\tilde{G}}(o) = 2,\quad\deg_{\tilde{G}}(\tilde{o}) = 1,\\ &\deg_{\tilde{G}}(v) \leq d+1\; \text{ for all } v \in \tilde{V}\backslash V,
	\end{align*}
	and such that the contact process on~$\tilde{G}$ satisfies the following properties.
	For all~$\lambda' \ge \lambda$,~$A \subset V$ and~$t > 0$,
	\begin{equation}
	\label{eq:pass}\mathbb{P}\left(\barCP[A]{\tilde{G}}{\lambda'}(\tilde{o}) > h \mid \barCP[A]{G}{\lambda'}(o) > t \right) > 1 - \exp\{-c_\lambda \cdot  t\} - \frac{1}{h},
	\end{equation}
	and, for all~$v \in V$,
	\begin{equation}
	\label{eq:pass}\mathbb{P}\left( \barCP[A]{\tilde{G}}{\lambda'}(v) > h  \mid \barCP[A]{G}{\lambda'}(o) > t \right) > 1 - \exp\{-c_\lambda \cdot  t\} - \frac{1}{h}.
	\end{equation}
	Moreover, for all~$\lambda' < \lambda$ there exists $\mathcal{H}' = \mathcal{H}'((G,o),\lambda, \lambda')$ such that
	\begin{equation}
	\label{eq:not_pass}\mathbb{P}\left(\barCP[A]{\tilde{G}}{\lambda'}(\tilde{o}) > 0 \right) < \exp \left\{ -d^{\sqrt{h}}\right\} \quad \text {for any } A \subset V \text{ and } h \geq \mathcal{H}'.
	\end{equation}
\end{theorem}

\begin{proof}[Proof of Theorem~\ref{thm:main}(a)]
Given a rooted tree~$(G,o)$ and~$\lambda > 0$, for each~$h \ge \mathcal{H}((G,o),\lambda)$, we denote by~$\mathcal{G}_h((G,o),\lambda)$ the rooted graph~$(\tilde{G},\tilde{o})$ corresponding to~$(G,o),\lambda,h$ as in Theorem~\ref{thm:tool}.

Fix~$\lambda \in (0,\lambda_c(\mathbb{Z}))$. Also fix an increasing sequence~$(\lambda'_n)$ with~$\lambda'_n \nearrow \lambda$. We will define an increasing sequence of graphs~$(G_n)$ by applying Theorem~\ref{thm:tool} repeatedly. We let~$G_0$ be a graph consisting of a single vertex (its root),~$o_0$. Once~$(G_n,o_n)$ is defined, fix
\begin{equation}\label{eq:choice_h1}h_{n+1} \ge \max\left(\mathcal{H}((G_n,o_n),\lambda),\; c_\lambda^{-1} (n+3) \log2,\;2^{n+3}\right)\end{equation}
and let~$(G_{n+1},o_{n+1}) := \mathcal{G}_{h_{n+1}}((G_n,o_n),\lambda)$. Increasing~$h_{n+1}$ if necessary, by~\eqref{eq:not_pass} we can also assume that
\begin{equation}\label{eq:no_pass1}\mathbb{P}\left(\barCP[A]{G_{n+1}}{\lambda_{n+1}'}(o_{n+1}) = 0 \right)> 1-\frac{1}{n}\quad  \text{ for any }A \subset G_n.\end{equation}

Note that $(G_n)_{n \in \mathbb{N}} = (V_n, E_n)_{n \in \mathbb{N}}$ is an increasing sequence of graphs in the sense that both $(V_n)_{n \in \mathbb{N}}$ and $(E_n)_{n \in \mathbb{N}}$ are increasing sequences of sets. Therefore we can define~$G_\infty = (V_\infty,E_\infty)$ where $V_\infty = \cup V_n$ and $E_\infty = \cup E_n$.

We will now show that~$G_\infty$ has the desired properties. Since each~$G_n$ is a tree,~$G_\infty$ is also a tree. The fact that~$G_\infty$ has bounded degree is  an immediate consequence of the degree conditions given in the end of the statement of Theorem~\ref{thm:tool}.

Let us verify that the contact process  with parameter~$\lambda$ on~$G$ survives locally. Start noting that
$$\mathbb{P}\left(\barCP[\{o_0\}]{G_0}{\lambda}(o_0) > c_\lambda^{-1} \log 4  \right) = 4^{-c_\lambda^{-1}}.$$
Next, using~\eqref{eq:pass} and~\eqref{eq:choice_h1},
$$\mathbb{P}\left(\barCP[\{o_0\}]{G_1}{\lambda}(o_1) > h_1\mid \barCP[\{o_0\}]{G_0}{\lambda}(o_0) > c_\lambda^{-1} \log 4 \right) > 1 - \frac14 - \frac{1}{h_1} \ge \frac12$$
and, for~$n \ge 1$,
$$
\mathbb{P}\left(\barCP[\{o_0\}]{G_{n+1}}{\lambda}(o_{n+1}) > h_{n+1} \mid \barCP[\{o_0\}]{G_{n}}{\lambda}(o_{n}) > h_{n}  \right) > 1 - \exp\{-c_\lambda h_n\} - \frac{1}{h_{n+1}} > 1 - \frac{1}{2^{n+1}}
$$
and similarly,
$$
\mathbb{P}\left(\barCP[\{o_0\}]{G_{n+1}}{\lambda}(o_{0}) > h_{n+1} \mid \barCP[\{o_0\}]{G_{n}}{\lambda}(o_{n}) > h_{n}  \right) >1 - \frac{1}{2^{n+1}}.
$$
From this, it follows that
$$\mathbb{P}\left(\barCP[\{o_0\}]{G}{\lambda}(o_0) = \infty \right) > 0,$$
so we have local survival at $\lambda$.

Now fix~$\lambda' < \lambda$; let us prove  that the contact process on~$G$ with parameter~$\lambda'$ dies out globally. By our construction of the graph $G$, survival of the infection implies in eventually infecting every $o_n$. However, by~\eqref{eq:no_pass1}, we have for any~$n$ such that~$\lambda_n' > \lambda'$,
%$$\mathbb{P}\left(\barCP[\{o_0\}]{G_{N}}{\lambda'}(o_{N}) = 0 \right) \ge \mathbb{P}\left(\barCP[\{o_0\}]{G_{N}}{\lambda_n'}(o_{N}) = 0 \right) \stackrel{\eqref{eq:not_pass}}{\ge} 1 - \sum_{i=n+1}^N 2^{-i}.$$
$$\mathbb{P}\left(\barCP[\{o_0\}]{G_{n}}{\lambda'}(o_{n}) = 0 \right) \ge \mathbb{P}\left(\barCP[\{o_0\}]{G_{n}}{\lambda_n'}(o_{n}) = 0 \right) \ge 1 - n^{-1}.$$
%Hence~$G$ has vertices that the contact process with parameter~$\lambda'$ does not reach with probability arbitrarily close to one. This implies that this process cannot survive locally.
%Thus, by Borel-Cantelli, we conclude that, with probability $1$, there exists $n \in \mathbb{N}$ for which $\barCP[\{o_0\}]{G_{N}}{\lambda'}(o_{n}) = 0$. By our construction of the graph $G$, it follows that the infection only spreads within a finite set and that the process hence dies out globally. Since this holds for every $\lambda' < \lambda$ we conclude that $\lambda_c^\mathrm{glob}(G) = \lambda_c^\mathrm{loc}(G)= \lambda$.
It follows that $o_n$ is never infected with high probability and that the process hence dies out globally. Since this holds for every $\lambda' < \lambda$ we conclude that $\lambda_c^\mathrm{glob}(G) = \lambda_c^\mathrm{loc}(G)= \lambda$.
\end{proof}
\begin{proof}[Proof of Theorem~\ref{thm:main}(b)]
We fix~$\lambda \in (0,\lambda_c(\mathbb{Z}))$ and again we will define an increasing sequence of graphs~$(G_n)$ by applying Theorem~\ref{thm:tool} repeatedly. Only now we take a decreasing sequence~$(\lambda'_n)$ with~$\lambda'_n \searrow \lambda$.  Like before we let~$G_0$ be a graph consisting of a single vertex (its root),~$o_0$ and, once~$(G_n,o_n)$ is defined, fix
\begin{equation}\label{eq:choice_h2}h_{n+1} \ge \max\left(\mathcal{H}((G_n,o_n),\lambda'_{n+1}),\; c_{\lambda_{n+2}} (n+3) \log2,\;2^{n+3}\right)
\end{equation}
and let~$(G_{n+1},o_{n+1}) := \mathcal{G}_{h_{n+1}}((G_n,o_n),\lambda'_{n+1})$. Since~$\lambda < \lambda_{n+1}$, increasing~$h_{n+1}$ if necessary, by~\eqref{eq:not_pass} we can assume that
\begin{equation}\label{eq:no_pass2}\mathbb{P}\left(\barCP[A]{G_{n+1}}{\lambda}(o_{n+1}) = 0 \right)> 1-\frac{1}{n} \quad  \text{ for any }A \subset G_n.\end{equation}
We then let~$G_\infty$ be the the limiting graph of the sequence $(G_n)_{n \in \mathbb{N}}$, as in (a). From this it follows that $G_\infty$ is a bounded degree tree.

The fact that the contact process with parameter $\lambda$ on $G_\infty$ dies out globally follows similarly to the last argument in the previous proof. Using \eqref{eq:no_pass2} gives
$$\mathbb{P}\left(\barCP[\{o_0\}]{G_{\infty}}{\lambda}(o_{n}) \neq 0 \right) = \mathbb{P}\left( \barCP[\{o_0\}]{G_n}{\lambda}(o_{n}) \neq 0 \right) \leq n^{-1} \quad \forall n.$$
The conclusion follows as in \ref{thm:main}(a).

Now, fix~$\lambda' > \lambda$, and take~$n$ such that~$\lambda'_n < \lambda'$. We then note that the event
$$\left\{\barCP[\{o_0\}]{G_n}{\lambda'_n}(o_n) > h_n\right\}$$
has positive probability, and that, for each~$N > n$, by~\eqref{eq:pass} and~\eqref{eq:choice_h2},
$$\mathbb{P}\left(\barCP[\{o_0\}]{G_N}{\lambda'_N}(o_N) > h_N \mid \barCP[\{o_0\}]{G_{N-1}}{\lambda'_n}(o_{N-1})> h_{N-1} \right) > 1 - \frac{1}{2^{N}},$$
and
$$\mathbb{P}\left(\barCP[\{o_0\}]{G_N}{\lambda'_N}(o_0) > h_N \mid \barCP[\{o_0\}]{G_{N-1}}{\lambda'_n}(o_{N-1}) > h_{N-1}\right) > 1 - \frac{1}{2^{N}}.$$
From this, local survival at parameter~$\lambda'$ follows as in part~(a).
\end{proof}

\section{Estimates for line segments and trees}
\label{s:line_tree}

This section is devoted to listing bounds for the behavior of the contact process on finite  trees and line segments which will be useful for our graph construction. 

Let us first mention two results that hold on general graphs. First, if~$G = (V,E)$ is a connected graph and~$x,y \in V$ and we let $\mathrm{dist}_G(x,y)$ denote the graph distance between $x$ and $y$ in $G$, we have
\begin{equation}\label{eq:lower_bound_reach}
\P \left(\CP[\{x\}]{G}{\lambda}{t}(y) = 1 \text{ for some } t \leq \mathrm{dist}_G(x,y) \right) \geq \left( e^{-2}(1-e^{-\lambda}) \right)^{\mathrm{dist}_G(x,y)}.
\end{equation} 
This is obtained by fixing a geodesic~$v_0 = x,\;v_1,\ldots, v_n=y$ (with~$n = \mathrm{dist}_G(x,y)$) and prescribing that, in each time interval~$[i,i+1]$ with~$0 \le i \le n-1$, there is no recovery mark at~$v_i$ or~$v_{i+1}$, and there is a transmission arrow from~$v_i$ to~$v_{i+1}$.

Second, we have the following inequality for the extinction time of the contact process on~$G$ started from full occupancy.
\begin{lemma}
\label{lem:inv_markov} 
For every $s > 0$, we have
\begin{equation}\label{eq:ud_markov}
\mathbb{P}\left( \CP[G]{G}{\lambda}{s} = \varnothing \right)  \leq \frac{s}{\mathbb{E}\left[\inf\left\{t: \CP[G]{G}{\lambda}{t} = \varnothing \right\} \right]}.
\end{equation}
\end{lemma}
This follows from noting that for any~$s$, the extinction of the process started  from full occupancy is stochastically dominated by the random variable $sX$, where $X$ has geometric distribution with parameter $\P(\CP[G]{G}{\lambda}{t} = \varnothing)$. See Lemma 4.5 in~\cite{mmvy} for a full proof.

\subsection{Contact process on line segments}
We will need some estimates involving the contact process on half-lines and line segments. From now on, we fix~$\lambda < \lambda_c(\mathbb{Z})$. The results below are essentially all consequences of the exponential bound
\begin{equation}\label{eq:exp_decay_time}
\P \left( \CP[\{0\}]{\Z}{\lambda}{t} \neq \varnothing \right) \leq \exp\{-c_\lambda \cdot t\},\quad t \ge 0
\end{equation}
for some~$c_\lambda > 0$; see Theorem~2.48 in Part~I of~\cite{lig2}. By simple stochastic comparison considerations and large deviation estimates for Poisson random variables, this also implies that
\begin{equation}\label{eq:crossing}
\P\left(\CP[\{0\}]{\Z}{\lambda}{t} = \varnothing,\; \bigcup_{s \le t}\CP[\{0\}]{\Z}{\lambda}{s} \subset [-t,t] \right) > 1-\exp\{-c_\lambda'\cdot  t\},\quad t \ge 0
\end{equation}
for some~$c_\lambda' > 0$.

For each~$\ell \in \mathbb{N}$, let~$\mathbb{L}_\ell$ denote the subgraph of~$\mathbb{Z}$ induced by the vertex set~$\{0,\ldots, \ell\}$. The following result is an immediate consequence of~\eqref{eq:exp_decay_time}, so we omit its proof.
\begin{lemma}\label{lem:more_than_log}
We have
\begin{equation}\label{eq:more_than_log}\lim_{\ell \to \infty} \mathbb{P}\left(\CP[\L_\ell]{\L_\ell}{\lambda}{(\log(\ell))^2} \neq \varnothing \right) = 0.\end{equation}
	\end{lemma}
Next, we bound the  probability of existence of an infection path starting  from a space-time point in the segment~$\{0\} \times [0,t]$ and crossing~$\L_\ell$.
\begin{lemma}\label{lem:exp_decay}
For any ~$\ell \in \N$ the contact process with parameter~$\lambda$ on~$\L_\ell$ satisfies
\begin{equation}\label{eq:exp_decay}
\mathbb{P}\big(\{0\} \times [0,t] \leadsto \{\ell\} \times [0, \infty)\big) \leq  e\cdot (t+1) \cdot \mathbb{P}\left((0,0) \rightsquigarrow \{\ell\} \times [0,\infty)\right),\quad t > 0.
\end{equation}
\end{lemma}
\begin{proof}
Define the event
$$A:= \left\{\{0\}\times [1,t+1]\rightsquigarrow \{\ell\} \times [0,\infty)\right\},$$	
so that the probability in the left-hand side in~\eqref{eq:exp_decay} is equal to~$\mathbb{P}(A)$. Let $X$ denote the Lebesgue measure of the random  set of times
 $$\{s \in [0,t+1]:\;(0,s)\rightsquigarrow \{\ell\} \times [0,\infty)\}.$$ 
Denote by $\mathcal{F}$ the $\sigma$-algebra generated by all the Poisson processes in the graphical construction of the contact process on~$\L_\ell$, and let~$\mathcal{F}'$ be similarly defined, except that it disregards all the recoveries marks at~$0$ that occur before time $t+1$. Note that~$X$ is measurable with respect to~$\mathcal{F}$ and $A \in \mathcal{F}'$. Moreover, we have
$$\mathbb{E}[X\mid \mathcal{F}'] \ge e^{-1} \text{ on }A,$$
since if~$A$ occurs and~$s \in [1,t+1]$ is such that~$(0,s)\rightsquigarrow \{\ell\} \times [0,\infty)$, then with probability~$e^{-1}$ there is no recovery mark on~$[s-1,s]$, so that~$X \ge 1$. We thus obtain
\begin{align*}
\P(A) \leq \mathbb{E}\left[e\cdot \E\left[X \mid \mathcal{F}'\right] \right] = e\cdot \mathbb{E}[X] &= e\int_0^{t+1} \P\left((0,s)\rightsquigarrow \{\ell\} \times [0,\infty)\right)\mathrm{d}s\\
&= e\cdot (t+1)\cdot \mathbb{P}\left((0,0) \rightsquigarrow \{\ell\} \times [0,\infty)\right).
\end{align*}
\end{proof}
The following corollary is a straightforward consequence of Lemma \ref{lem:exp_decay} and \eqref{eq:crossing}.
\begin{corollary}\label{cor:exp_decay_space}
There exists~$c_\L > 0$ such that, for~$\ell \in \N$ large enough, the contact process with parameter~$\lambda$ on~$\L_\ell$ satisfies
\begin{equation}\label{eq:exp_decay_space}
\mathbb{P}\big(\{0\} \times [0,t] \leadsto \{\ell\} \times [0, \infty)\big) \leq (t+1) \exp \{-c_\L \ell \}.
\end{equation}
\end{corollary}
We now show that the subcritical contact process on~$\Z$ started from occupation in a half-line~$\{1,2,\ldots\}$ has positive probability of never infecting the origin.
\begin{lemma}\label{basic_line}
There exists $c_{\L} > 0$ such that
\begin{equation}
\P \left( \CP[\{1,2,\ldots\}]{\Z}{\lambda}{t}(0) = 0 \text{ for all } t\right) > c_{\L}.
\label{eq:never_infects}\end{equation}
\end{lemma}
\begin{proof}
For~$n \in \N$, let~$A(n)$ denote the event that vertices~$1,\ldots,n-1$ have a recovery mark and generate no transmission arrow in the time interval~$[0,1]$. We have
$$\P \left( \CP[\{1,2,\ldots\}]{\Z}{\lambda}{t}(0) = 0 \text{ for all } t\right)  \ge \P(A(n)) \cdot \P\left(\CP[\{n,n+1,\ldots\}]{\Z}{\lambda}{t}(0) = 0 \text{ for all }t\right). $$
For any~$n \in \N$, we have~$\P(A(n)) > 0$ and
\begin{align*}\P\left(\CP[\{n,n+1,\ldots\}]{\Z}{\lambda}{t}(0) = 0 \text{ for all }t\right) &\ge \mathbb{P}\left(\bigcap_{i=n}^\infty \left\{ \CP[\{i\}]{\Z}{\lambda}{t} \subset [i/2,3i/2] \text{ for all } t \right\} \right)\\
&\stackrel{\eqref{eq:crossing}}{\ge} 1 - \sum_{i=n}^\infty \exp\{-c'_\lambda \cdot  i\}, \end{align*}
which can be made positive by taking~$n$ large enough.
\end{proof}

Finally, we compare the contact process on the same graph for two different values of the infection parameter.

\begin{lemma}\label{lem:line_lambda}
For all $\lambda', \lambda > 0$ with $\lambda' < \lambda < \lambda_c(\mathbb{Z})$ there exists $\eta = \eta_{\lambda',\lambda} > 1$ such that, for $\ell$ large enough,
\begin{equation}\label{eq:line_lambda}
\P \left( \barCP[\{0\}]{\{0,1,...\}}{\lambda'}(\ell) > 0 \right) \leq \eta^{-\ell} \P \left( \barCP[\{0\}]{\{0,1,...\}}{\lambda}(\ell) > 0 \right) 
\end{equation}
\end{lemma}
\begin{proof}
Using monotonicity and the Markov property it can be proved that the limit
$$\beta(\lambda) = \lim \P \left( \barCP[\{0\}]{\Z}{\lambda}(\ell) > 0 \right)^{1/\ell}$$
exists (see discussion preceding Proposition 4.50 in \cite{lig2} for a full proof of this fact). Furthermore, it was shown in \cite{lal2} that, for the contact process on a regular tree, if
$$\lambda' < \lambda \text{ and } \beta(\lambda) < 1 / \sqrt{d}$$
then $\beta(\lambda') < \beta(\lambda)$. Noting that the exponential bound \eqref{eq:exp_decay_time} implies that $\beta(\lambda_c(\Z)) < 1$, we have the result for the contact process on $\Z$. Finally, \cite{lal1} proves that
$$\lim \P \left( \barCP[\{0\}]{\Z}{\lambda}(\ell) > 0 \right)^{1/\ell} = \lim \P \left( \barCP[\{0\}]{\{0,1,...\}}{\lambda}(\ell) > 0 \right)^{1/\ell}.$$
\end{proof}

\subsection{Contact process on finite trees}

To conclude this section, we gather a few estimates from~\cite{cmmv} concerning the contact process on finite trees.  We continue with fixed~$\lambda < \lambda_c(\mathbb{Z})$, and assume~$d$ is large enough that~$\lambda > \lambda_c^\mathrm{loc}(\mathbb{T}^d)$. For each~$h \in \mathbb{N}$, we let~$\T^d_h$ be a rooted tree with branching number~$d$, truncated at height~$h$. This  means that~$\T^d_h$ is a tree with a root vertex~$\rho$ with degree~$d$, and so that vertices at graph distance between one and~$h-1$ from~$\rho$ have degree~$d+1$, and vertices at graph distance~$h$ from~$\rho$ have degree one. The next result contains two statements concerning the contact process on~$\mathbb{T}^d_h$. First, the process started from full occupancy survives for a time at least as large as exponential in~$d^h$, with probability tending to one exponentially in~$d^h$. Second, with probability bounded away from zero, the process started from any non-empty configuration couples with the process started from full occupancy within time~$\exp\{d^{h^{1/5}}\}$ (and both process remain alive at a time that is exponential  in~$d^h$).

\begin{proposition}\label{basic_trees}
	There exists $c_{\T} = c_{\T}(\lambda,d) > 0$ such that, for $h$ large enough,
	\begin{equation}\label{eq:metastability}
	\mathbb{P}\left(\CP[\T^d_h]{\T^d_h}{\lambda}{\exp\{c_\T\cdot d^h\}} \neq \varnothing \right) > 1-\exp\{-c_{\T}\cdot {d^h}\}
	\end{equation}
	and, letting~$\mathcal{t}(h):= \exp\{d^{h^{1/5}}\}$,
	\begin{equation}\label{eq:surv_subset}
	\inf_{\substack{A \subset \T^d_h,\\A \neq \varnothing}}\mathbb{P} \left(\CP[A]{\T^d_h}{\lambda}{\mathcal{t}(h)} = \CP[\T^d_h]{\T^d_h}{\lambda}{\mathcal{t}(h)},\;\; \CP[A]{\T^d_h}{\lambda}{\exp\left\{c_\T\cdot d^h \right\}} \neq \varnothing \right) > c_{\T}.
	\end{equation}
\end{proposition}
\begin{proof}
	Theorem~1.5 in~\cite{cmmv} states that the limit
	\begin{equation}\label{eq:background_tree1}
	\lim_{h \to \infty}\frac{\log \mathbb{E}\left[\inf\left\{t: \CP[\T^d_h]{\T^d_h}{\lambda}{t} = \varnothing \right\} \right]}{|\T^d_h|}
	\end{equation}
	exists and is positive; denote it by~$c_1$. Taking~$c_\T < c_1/4$, the inequality~\eqref{eq:metastability} follows from this combined with~\eqref{eq:ud_markov}. Next,  Corollary~4.10 in~\cite{cmmv} implies that there exists a constant~$c_2 > 0$ such that
	$$\inf_{\substack{A \subset \T^d_h,\\A \neq \varnothing}}\mathbb{P}\left(\CP[A]{\T^d_h}{\lambda}{\mathcal{t}(h)} \neq \varnothing \right) > c_2,$$
	and Proposition~4.15 in~\cite{cmmv} gives
	$$\sup_{\substack{A \subset \T^d_h,\\A \neq \varnothing}} \mathbb{P}\left(\CP[A]{\T^d_h}{\lambda}{\mathcal{t}(h)} \neq \varnothing,\;\CP[A]{\T^d_h}{\lambda}{\mathcal{t}(h)} \neq \CP[\T^d_h]{\T^d_h}{\lambda}{\mathcal{t}(h)} \right) \xrightarrow{h \to \infty} 0.$$
	Using these two facts and also~\eqref{eq:metastability}, we obtain that, if~$c_\T < \min(c_1/4,c_2/2)$, then for any~$A \subset \T^d_h$,~$A \neq \varnothing$,
	$$\mathbb{P}\left(\CP[A]{\T^d_h}{\lambda}{\exp\{c_\T \cdot d^h\}} \right) \ge \mathbb{P}\left(\varnothing \neq \CP[A]{\T^d_h}{\lambda}{\mathcal{t}(h)}  = \CP[\T^d_h]{\T^d_h}{\lambda}{\mathcal{t}(h)},\;\; \CP[\T^d_h]{\T^d_h}{\lambda}{\exp\{c_\T \cdot d^h\}} \neq \varnothing \right) > c_\T.$$
\end{proof}

\section{Proof of Theorem~\ref{thm:tool}}
\label{sec:proof_thm}
In this section, we will give some key definitions and state three results (Propositions~\ref{prop:ignition},~\ref{prop:time_spent} and~\ref{prop:no_go}) that will immediately imply Theorem~\ref{thm:tool}.  The idea of our graph augmentation~$(\tilde{G},\tilde{o})$ of a given rooted graph~$(G,o)$ is summarized by Figure~\ref{fig:gtilde} below: next to the root~$o$ of~$G$, we append a copy of~$\mathbb{T}^d_h$ (with~$h$ large), followed by a line segment whose length is a function of~$h$, denoted~$L(h)$. The endpoint of this line segment that is away from the tree is the root~$\tilde{o}$ of~$\tilde{G}$. We will be free to take~$h$ large (adjusting the length~$L(h)$ accordingly) so as to guarantee several desirable properties for~$\tilde{G}$.

Now let us briefly discuss the differences between the construction carried out in this work and the desert-oasis construction in \cite{ss2}. First note that in \cite{ss2} the goal is to show that the process survives at~$\lambda_c(\mathbb{Z})$ and dies out at any~$\lambda < \lambda_c(\mathbb{Z})$. At each step of their construction, they take the oasis-desert pair satisfying two things. First, the oasis is a structure that sustains for a long time the process with rate $\lambda_c(\mathbb{Z})$ (and possibly also with slightly smaller rates). Second, the volume of the oasis is small compared to the length of the desert: in the~$n$th augmentation of their graph construction, the oasis has volume~$n^2$ and the desert has  length~$n^3$. This way, if~$\lambda < \lambda_c(\mathbb{Z})$, even if the process survives for time~$\exp\{c(\lambda)n^2\}$ in the oasis, this is not enough to overcome the probability of crossing the oasis, which is around~$\exp\{-c'(\lambda)n^3\}$. However, in our case, we want to have as target rate a value~$\lambda$ that is already subcritical, so it is not so clear which sizes to put in place  of~$n^2$ and~$n^3$, or how a given choice of sizes can be favorable for the target~$\lambda$ but not for smaller values. Our solution involves defining the length of the desert implicitly (depending on~$\lambda$ and on a fixed oasis size), in a way that the the probability of crossing the desert starting from the oasis at rate~$\lambda$ is around a prescribed value.

Fix~$\lambda \in (0,\lambda_c(\mathbb{Z}))$. The value~$d = d_\lambda$ that appears in the statement of Theorem~\ref{thm:tool} will now be chosen:~$d$ should  be  large enough that~$\lambda > \lambda_c^{\mathrm{loc}}(\mathbb{T}^d)$, and also
\begin{equation}
\label{eq:choice_of_d} m = m_\lambda:= d\cdot (1-e^{-\lambda}) \cdot e^{-2} > 1.
\end{equation}
From now on, we fix~$(G,o)=((V,E),o)$ a rooted tree with degrees bounded by~$d+1$ and with~$\deg_G(o)= 1$, as in the statement of Theorem~\ref{thm:tool}.

Throughout this section, it will be useful to abbreviate
\begin{equation}
\label{eq:cals}
\cals(h) := \exp \left\{ d^{\sqrt{h}}\right\},\quad h \in \mathbb{N}.
\end{equation}

We first define an auxiliary graph~$\hat{G}$, depending on~$(G,o)$ and on a positive integer~$h$ (which we often omit from the notation), as follows. We let~$\mathcal{T}_h$ be a copy of~$\mathbb{T}^d_h$, with root~$\rho$, and let~$\mathcal{L}_\infty$ be a half-line with extremity  denoted~$v_-$. We then let~$\hat{G}$ denote the graph obtained by putting the three graphs~$G, \mathcal{T}_h,\mathcal{L}_\infty$ together, and connecting them by including an edge between~$o$ (the root of~$G$) and~$\rho$ (the root of~$\mathcal{T}_h$), and an edge between~$\rho$ and~$v_-$ (the extremity of~$\mathcal{L}_\infty$).

For each~$\ell \in \mathbb{N}_0$, let~$v_\ell$ denote the vertex of~$\mathcal{L}_\infty$ at distance~$\ell$ from~$v_-$ (in particular,~$v_0 = v_-$), and define
\begin{equation}\label{eq:def_of_p_cross}\mathcal{P}(\ell) = \mathcal{P}_{(G,o),h}(\ell) := \mathbb{P}\left( \bar{\xi}^{V \cup \mathcal{T}_h}_{\hat{G},\lambda}(v_\ell) > 0\right),\end{equation}
that is,~$\mathcal{P}(\ell)$ is the probability that~$v_\ell$  becomes infected in the contact process on~$\hat{G}$ with parameter~$\lambda$ and initial configuration~$V \cup \mathcal{T}_h$. Note that~$\mathcal{P}$ is non-increasing. 
\begin{lemma}[Properties of~$\mathcal{P}$]
\label{lem:well_def} We have
$$\lim_{\ell \to \infty} \mathcal{P}(\ell) = 0$$ and, if~$h$ is large enough,
\begin{equation}\label{eq:first_bound_P}\mathcal{P}(0), \mathcal{P}(1) \ge 1 - \cals(h)^{-1}.\end{equation}
\end{lemma}
\begin{proof} The first statement follows from the fact that the contact process with parameter~$\lambda$ on~$\hat{G}$ dies out (which is in turn an easy consequence of the facts that~$G, \mathcal{T}_h$ are finite graphs and~$\lambda < \lambda_c(\mathbb{Z})$).

For the second statement, we will only treat~$\mathcal{P}(0)$, since the proof for~$\mathcal{P}(1)$ is the same. Assume~$h$ is larger than the graph diameter of~$G$. Then, for any non-empty~$A \subset V \cup \mathcal{T}_h$ we have
$$\mathbb{P}\left(\CP[A]{\hat{G}}{\lambda}{t}(v_0) = 1 \text{ for some } t \leq h \right) \ge \left(e^{-1}\cdot (1-e^{-\lambda}) \right)^h.$$
Iterating this, we obtain
\begin{align*}&\mathbb{P}\left(\xi^A_{\hat{G},\lambda;\cals(h)} \neq \varnothing,\;  \CP[A]{\hat{G}}{\lambda}{t}(v_0) = 0\; \forall t \leq \cals(h) \right)\le \left(1- \left(e^{-1}\cdot (1-e^{-\lambda})\right)^h \right)^{\lfloor \cals(h)/h\rfloor}.\end{align*}
The result now follows from noting that the right-hand side above is much smaller than~$\cals(h)^{-1}$, and moreover
$$\mathbb{P}\left(\CP[V \cup \mathcal{T}_h]{\hat{G}}{\lambda}{\cals(h)} = \varnothing \right) \le \mathbb{P}\left( \CP{\mathbb{T}^d_h}{\lambda}{\cals(h)} = \varnothing \right) \le \frac{\cals({h})}{\exp\{c_\mathbb{T} \cdot d^h\}} \ll \cals(h)^{-1}$$
if~$h$ is large, where in the second inequality we have used Lemma \ref{lem:inv_markov} and Proposition \ref{basic_trees}.
\end{proof}
With the above result at hand, for~$h$ large enough we can define
\begin{equation}\label{eq:def_of_L_cross}L(h) := \inf\left\{\ell \in \mathbb{N}_0:  \mathcal{P}(\ell) < 1-\cals(h)^{-1} \right\}\end{equation}
and have~$L(h) > 1$. We now define the graph~$\tilde{G}$ in the same way as~$\hat{G}$, with the sole exception that, instead of the half-line~$\mathcal{L}_\infty$, it includes a line segment~$\mathcal{L}_h$ with vertex set
$$v_0 = v_-,\;v_1,\;\ldots,\; v_{L(h)}$$
(as before, we link~$\mathcal{L}_h$ to~$\mathcal{T}_h$ with an edge between~$\rho$ and~$v_-$). We denote by~$\tilde{V}$ and~$\tilde{E}$ the vertex and edges sets of~$\tilde{G}$, respectively. The vertex~$v_{L(h)}$ is the root of~$\tilde{G}$, denoted~$\tilde{o}$. The definition of~$(\tilde{G},\tilde{o})$ depends on~$(G,o)$ and~$h$, but this dependence will be omitted from the notation. We will several times assume that~$h$ is large (possibly depending on~$G$).

\begin{figure}[htb]
	\begin{center}
		\setlength\fboxsep{0pt}
		\setlength\fboxrule{0.0pt}
		\fbox{\includegraphics[width = 0.8\textwidth]{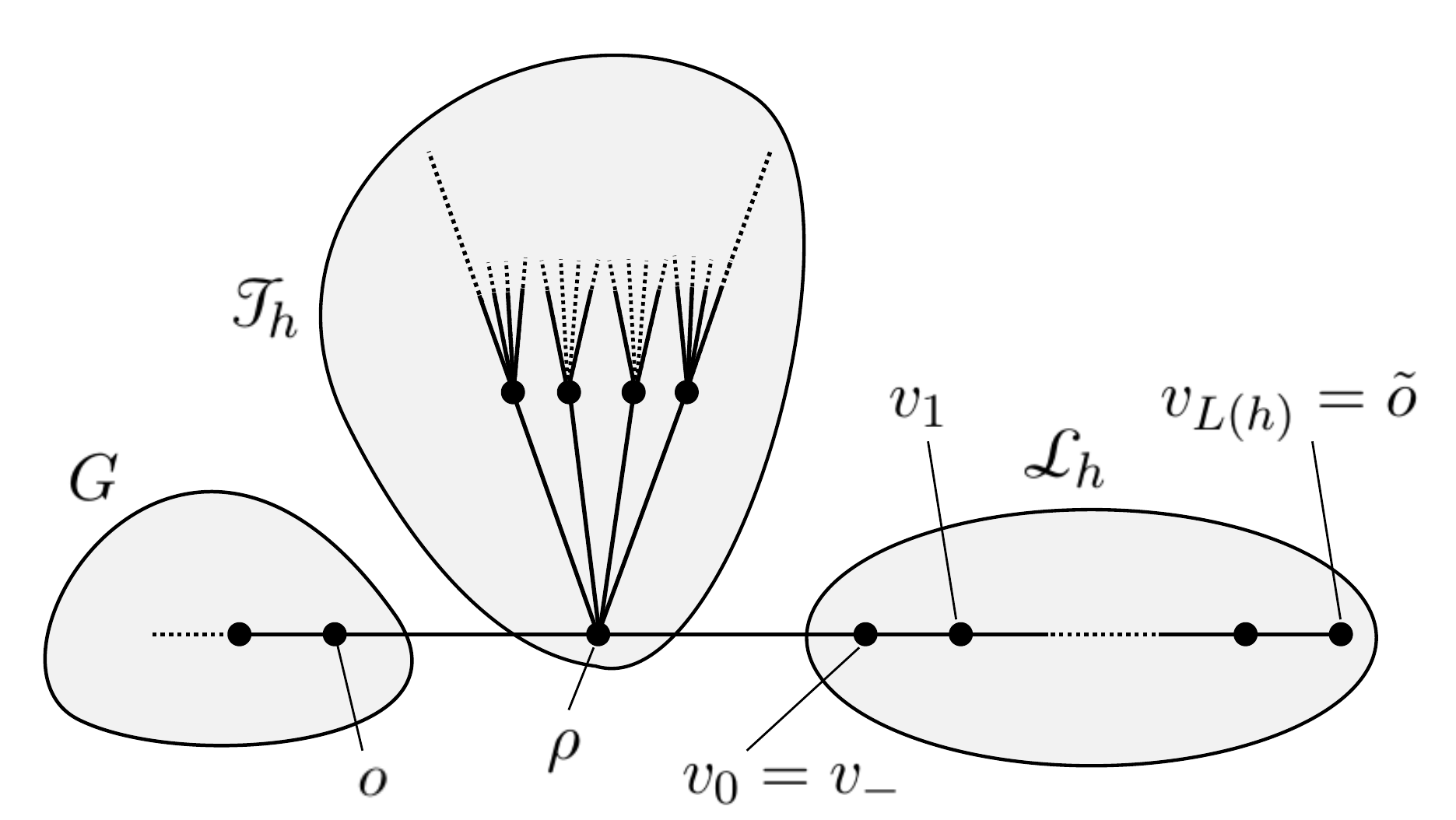}}
	\end{center}
	\caption{The graph~$\tilde{G}$.}
	\label{fig:gtilde}
\end{figure}

We will now state several results about~$(\tilde{G},\tilde{o})$, culminating in the proof of Theorem~\ref{thm:tool}. Define the set of configurations
$$\mathcal{A}_h := \left\{A \subset \tilde{V}: \#\{v \in A \cap \mathcal{T}_h: \mathrm{dist}_{\tilde{G}}(\rho,v) = \lfloor h/2\rfloor \}  \ge (m/2)^{\lfloor h/2\rfloor} \right\},$$
that is,~$A \in \mathcal{A}_h$ if~$A$ has at least~$m^{\lfloor h/2\rfloor}$ vertices at height~$\lfloor h/2\rfloor$ in~$\mathcal{T}_h$. The following result is the main reason for the introduction of~$\mathcal{A}_h$.
\begin{lemma}[Persistence starting from~$\mathcal{A}_h$] \label{lem:persist0}
If~$A \in \mathcal{A}_h$, then
$$\mathbb{P}\left(\CP[A]{\tilde{G}}{\lambda}{\cals(h)} \neq \varnothing \right) > 1 - \cals(h)^{-2}.$$
\end{lemma}
\begin{proof}
Fix~$A \in \mathcal{A}_h$ and let~$T_1,\ldots, T_{(m/2)^{\lfloor h/2\rfloor}}$ be disjoint copies of~$\mathbb{T}^d_{\lfloor h/2\rfloor}$ that appear as subtrees of~$\mathcal{T}_h$, rooted at a vertices~$v_1,\ldots, v_{(m/2)^{\lfloor h/2\rfloor}} \in A \cap \mathcal{T}_h$ at distance~$\lfloor h/2 \rfloor$ from~$\rho$. We have
\begin{align*}\mathbb{P}\left(\CP[A]{\tilde{G}}{\lambda}{\cals(h)} \neq \varnothing \right) &\ge \mathbb{P}\left(\bigcup_{i=1}^{(m/2)^{\lfloor h/2\rfloor}} \left\{\CP[\{v_i\}]{T_i}{\lambda}{\exp\{c_\mathbb{T}\cdot d^{\lfloor h/2\rfloor} \}} \neq \varnothing \right\} \right) \\[.2cm]& \stackrel{\eqref{eq:surv_subset}}{\ge} 1-\left(1- c_\mathbb{T} \right)^{(m/2)^{\lfloor h/2\rfloor}} \gg 1-\cals(h)^{-1}.\end{align*}
\end{proof}

\begin{proposition}[Ignition]\label{prop:ignition}
There exists~$c_\lambda > 0$ such that for~$h$ large enough, any~$\lambda' \ge \lambda$ and any~$A \subset V$ we have
\begin{equation*}
\mathbb{P}\left(\CP[A]{\tilde{G}}{\lambda'}{s} \in \mathcal{A}_h \text{ for some }s \ge 0\mid \barCP[A]{G}{\lambda'}(o) > t \right) > 1 - \exp\{-c_\lambda \cdot t\},
\end{equation*}
that is, given that the contact process with rate~$\lambda'$, started from~$A$ and confined to~$G$ spends more than~$t$ time units with~$o$ occupied, the probability that the same process on the full graph~$\tilde{G}$ reaches~$\mathcal{A}_h$ is higher than~$1-\exp\{-c_\lambda \cdot t\}$. 
\end{proposition}
We interpret the conditioning in the above statement as saying that the confined process has time~$t$ to attempt to ``ignite'' the infection on the tree~$\mathcal{T}_h$ (meaning fill it up sufficiently to enter the set~$\mathcal{A}_h$). We postpone the proof of this proposition to Section~\ref{ss:ignite}.

\begin{proposition}[From~$\mathcal{A}_h$ to~$\tilde{o}$]\label{prop:time_spent}
If~$h$ is large enough, then for any~$A \in \mathcal{A}_h$ we have
$$\mathbb{P}\left(\bar{\xi}^A_{\tilde{G},\lambda}(\tilde{o}) > h \right) > 1-\frac{1}{h} \quad \text{and} \quad\mathbb{P}\left(\bar{\xi}^A_{\tilde{G},\lambda}(v) > h \right) > 1-\frac{1}{h} \; \text{ for all }v \in V. $$
\end{proposition}
The proof of this proposition will be carried out in Section~\ref{ss:time_spent}.

\begin{proposition}\label{prop:no_go} For any~$\lambda' < \lambda$, if $h$ is large enough depending on $\lambda'$, then
%\begin{equation}\mathbb{P}\left(\barCP[A]{\tilde{G}}{\lambda'}(\tilde{o}) > 0 \right) < \cals(h)^3\cdot (\eta_{\lambda,\lambda'})^{d^{\frac{3h}{4}}}.\label{eq:bound_Pl}\end{equation}
\begin{equation}\mathbb{P}\left(\barCP[A]{\tilde{G}}{\lambda'}(\tilde{o}) > 0 \right) < \cals(h)^{-1} \quad \text{for any } A \subset V \cup \mathcal{T}_h.\label{eq:bound_Pl}\end{equation}
\end{proposition}
The proof of this proposition will be done in Section~\ref{ss:lower_lambda}.

\begin{proof}[Proof of Theorem~\ref{thm:tool}]
It follows from the construction that~$\tilde{G}$ satisfies the stated degree properties. The inequality~\eqref{eq:pass} follows from Propositions~\ref{prop:ignition} and~\ref{prop:time_spent}, and~\eqref{eq:not_pass} follows from Proposition~\ref{prop:no_go}.
\end{proof}

\section{Proofs of results in Section~\ref{sec:proof_thm}}
\label{s:proofs_of_props}
We now turn to the proofs of the three propositions of the previous section. In Section~\ref{ss:ignite}, we will prove Proposition~\ref{prop:ignition}. In Section~\ref{ss:preliminary}, we will give some bounds involving the function~$L(h)$, as well as a key proposition involving coupling of the contact process on~$\tilde{G}$, Proposition~\ref{prop:persist_couple}. Next, Section~\ref{ss:time_spent} contains the proof of Proposition~\ref{prop:time_spent}, and Section~\ref{ss:lower_lambda} contains the proof of Proposition~\ref{prop:no_go}.

\subsection{Proof of Proposition~\ref{prop:ignition}}
\label{ss:ignite}
\begin{proof}[Proof of Proposition~\ref{prop:ignition}]
	We begin with some definitions. For~$0 \le i \le h$, let~$T(i)$ denote the set of vertices of~$\mathcal{T}_h$ at distance~$i$ from the root~$\rho$. Using the graphical construction of the contact process with parameter~$\lambda' \ge \lambda$, we will now define random sets~$\mathcal{Z}_{\lambda'}(0),\ldots, \mathcal{Z}_{\lambda'}({\lfloor  h/2\rfloor})$ with~$\mathcal{Z}_{\lambda'}(i)\subset T(i)$ for each~$i$. We set~$\mathcal{Z}_{\lambda'}(0) := \{\rho\}$. Assume that~$\mathcal{Z}_{\lambda'}(i)$ has been defined, let~$z$ be a vertex of~$T(i+1)$ and let~$z'$ be the neighbour of~$z$ in~$T(i)$. We include~$z$ in~$\mathcal{Z}_{\lambda'}({i+1})$ if~$z' \in \mathcal{Z}_{\lambda'}(i)$ and, in the time interval~$[i,i+1]$, there are no recovery marks on~$z'$ or~$z$, and there is a transmission arrow from~$z'$ to~$z$. Letting~$Z_{\lambda'}(i):= |\mathcal{Z}(i)|$ for each~$i$, it is readily seen that~$(Z_{\lambda'}(i): 0 \leq i \leq \lfloor h/2\rfloor)$ is a branching process. Its offspring distribution is equal to the law of~$U \cdot W$, where~$U \sim \text{Bernoulli}(e^{-1})$ and~$W \sim \text{Binomial}(d,e^{-1}\cdot (1-e^{-\lambda'}))$ are independent. The expectation of this distribution is larger than~$m_\lambda > 1$. For this reason, there exists~$\sigma_\lambda > 0$ such that the event
	$$B_{\lambda'} := \left\{Z_{\lambda'}(\lfloor h/2 \rfloor) > (m_\lambda/2)^{\lfloor h/2 \rfloor}\right\}$$
	has
	\begin{equation*}
	\mathbb{P}\left( B_{\lambda'} \right) > \sigma_\lambda \quad \text{ for all }\lambda' \ge \lambda \text{ and }h \in \mathbb{N}.
	\end{equation*}
	Finally note that
	$$B_{\lambda'} \subset \left\{\CP[\{\rho\}]{\mathcal{T}_h}{\lambda'}{\lfloor h/2 \rfloor} \in \mathcal{A}_h\right\}.$$
	
	Now define~$B_{\lambda'}(0) := B_{\lambda'}$ and, for~$t \in [0,\infty)$, define~$B_{\lambda'}(t)$ as the time translation of~$B_{\lambda'}$, so that time~$t$ becomes the time origin (that is,~$B_{\lambda'}(t)$ is defined by using the graphical construction of the contact process on the time intervals~$[t,t+1]$,~$[t+1,t+2]$,~$\ldots$,~$[t+\lfloor h/2 \rfloor -1, t+ \lfloor h/2 \rfloor]$). We evidently have
	\begin{equation}\label{eq:indep_trans}\mathbb{P}(B_{\lambda'}(t)) = \mathbb{P}(B_{\lambda'}) > \sigma_\lambda \quad \text{ for any~$t$,}\end{equation} and moreover,
	\begin{equation}
	\label{eq:on_trans} \left\{\rho \in \CP[A]{\tilde{G}}{\lambda'}{t} \right\} \cap B_{\lambda'}(t)  \subset \left\{\CP[A]{\tilde{G}}{\lambda'}{t + \lfloor h/2\rfloor} \in \mathcal{A}_h \right\}
	\end{equation}
	for any~$A$. It will be useful to note that, if~$t_1, t_2 \ge 0$ with~$t_2 > t_1 + 2$, then~$B_{\lambda'}(t_1)$ and~$B_{\lambda'}(t_2)$ are independent.
	
	Now, fix~$t>0$ and condition on the event~$\left\{ \barCP[A]{G}{\lambda'}(o) > t\right\}$ occurs. Note that this  event only involves the graphical construction of the contact process on~$G$; in particular, the Poisson processes involving vertices and edges of~$\mathcal{T}_h$, or the edge~$\{o,\rho\}$, are still unrevealed. Then, by elementary properties of Poisson processes, there exists~$c_\lambda > 0$ (depending only on~$\lambda$) such that (uniformly on~$\lambda' \ge \lambda$) outside probability~$\exp\{-c_\lambda\cdot t\}$, we can find random times~$s_1 < \ldots < s_{\lfloor c_\lambda t \rfloor}$ separated from each other by more than two units, and such that for each~$i$,~$o \in \CP[A]{G}{\lambda'}{s_i}$ and there is a transmission arrow from~$(o,s_i)$ to~$(\rho,s_i)$. If this is the case, and if~$B_{\lambda'}(s_i)$ also occurs for some~$i$, we then get~$\CP[A]{\tilde{G}}{\lambda'}{s_i + \lfloor h/2\rfloor} \in \mathcal{A}_h$, by~\eqref{eq:on_trans}. The desired result now follows from independence between the events~$B_{\lambda'}(s_i)$, together with~\eqref{eq:indep_trans} and a Chernoff bound.
	\end{proof}

\subsection{Preliminary bounds}
\label{ss:preliminary}
In this section we will prove that $$ d^{\frac{3h}{4}} \le L(h) \le d^{2h}$$ for large enough $h$. These bounds will be instrumental for proving Propositions \ref{prop:time_spent} and \ref{prop:no_go}. We first give an upper bound involving the extinction time of the contact process on~$\tilde{G}$, in terms of the length~$L(h)$.

\begin{lemma} \label{lem:extt} We have
	$$\lim_{h \to \infty} \mathbb{P}\left( \CP[\tilde{V}]{\tilde{G}}{\lambda}{\exp\{d^{\frac32h}\}\cdot (\log L(h))^2}\neq \varnothing \right) = 0,$$
	that is, the extinction time of the contact process on~$\tilde{G}$ started from full occupancy is smaller than~$\exp\{d^{\frac32h}\}\cdot (\log L(h))^2$ with high probability as~$h \to \infty$.
\end{lemma}
\begin{proof}
	Let~$E_{0}'$ be the event that each vertex in~$V \cup \mathcal{T}_h$ has a recovery mark before it sends out any transmission arrow, and before time~$1$. Since all vertices of~$V \cup \mathcal{T}_h$ have degree at most~$d+1$, we have
	$$\mathbb{P}(E_0') \ge \left((1-e^{-1})\cdot e^{-(d+1)\lambda}\right)^{|V \cup \mathcal{T}_h|} \ge  \left((1-e^{-1})\cdot e^{-(d+1)\lambda}\right)^{d^{h+2}},$$
	if~$h$ is large enough (since~$|\mathcal{T}_h| < d^{h+1}$ and~$|V|$ is fixed as~$h \to \infty$). Next, let~$E_0''$ denote the event that the contact process on~$\tilde{G}$ started from~$\mathcal{L}_h$ infected dies out before time~$(\log L(h))^2$, and never infects the root~$\rho$ of~$\mathcal{T}_h$. That is,
	$$E_0'' := \left\{\CP[\mathcal{L}_h]{\tilde{G}}{\lambda}{(\log L(h))^2} = \varnothing,\; \rho \notin \CP[\mathcal{L}_h]{\tilde{G}}{\lambda}{t} \text{ for all }t  \right\}.$$
	The probability of~$E_0''$ is the same as the probability that a contact process on the line segment~$\{-1,0,\ldots, L_G(h)\}$, with rate~$\bar{\lambda}$ and initial configuration~$\{0,\ldots, L(h)\}$, dies out before time~$(\log L(h))^2$ and never infects vertex~$-1$. Therefore, by~\autoref{lem:more_than_log} and~\autoref{basic_line}, we have
	$$\mathbb{P}(E_0'') > \delta > 0$$
	for all~$h$. Let~$E_0 := E_0' \cap E_0''$; since~$E_0'$ and~$E_0''$ have different supports in the graphical representation, they are independent and hence
	$$\mathbb{P}(E_0) >  \delta\left((1-e^{-1})\cdot e^{-(d+1)\lambda}\right)^{d^{h+2}}.$$
	
	For~$i \in \{1,\ldots,\lfloor \exp\{d^{\frac32h}\} \rfloor\}$, let~$E_i$ be the time translation of event~$E_0$ to the graphical construction on the time interval
	$$[i(\log L(h))^2,\;(i+1)(\log L(h))^2].$$
	Finally, noting that~$E_0,E_1,\ldots$ are independent and
	$$E_i \subset \left\{\CP[\tilde{V}]{\tilde{G}}{\lambda}{(i+1)(\log L(h))^2} = \varnothing \right\},$$
	we have
	\begin{align*}&\mathbb{P}\left(\CP[\tilde{V}]{\tilde{G}}{\lambda}{\exp\{d^{\frac32h}\}\cdot(\log L(h))^2} \neq \varnothing \right) \\&\leq \mathbb{P}((E_0)^c)^{\lfloor \exp\{d^{\frac32h}\} \rfloor}\\&\leq \exp\left\{ -\lfloor \exp\{d^{\frac32h}\} \rfloor \cdot  \delta\left((1-e^{-1})\cdot e^{-(d+1)\lambda}\right)^{d^{h+2}} \right\} \xrightarrow{h \to \infty}0.\end{align*}
\end{proof}

We now proceed to an upper bound on~$L(h)$.
\begin{lemma}\label{prop:bounds}
If~$h$ is large enough we have
\begin{equation}\label{eq:main_bounds}
L(h) \le d^{2h}
\end{equation}
\end{lemma}

\begin{proof}
Define
$$F_1:= \left\{\CP[V \cup \mathcal{T}_h]{\tilde{G}}{\lambda}{\exp\{d^{\frac32h}\}\cdot (\log L(h))^2} = \varnothing\right\}.$$
Recall that~$v_{L(h)-1}$ denotes the neighbor of~$\tilde{o}$ in~$\mathcal{L}_h$, and let~$F_2$ be the event that there is no infection path starting from~$(\rho,s)$ for some~$s \le \exp\{d^{\frac32h}\}\cdot (\log L(h))^2$, ending at~$(v_{L(h)-1},t)$ for some~$t > s$, and entirely contained in~$\mathcal{L}_h \cup \{\rho\}$. It is easy to see that
$$F_1 \cap F_2 \subset \left\{ \CP[V \cup \mathcal{T}_h]{\tilde{G}}{\lambda}{t}(v_{L(h)-1}) = 0 \; \text{for all }t \ge 0 \right\}.$$
By Lemma~\ref{lem:extt} we have~${\displaystyle \lim_{h \to \infty} \mathbb{P}(F_1) = 1}$
and by~\autoref{cor:exp_decay_space} we have
$$\mathbb{P}(F_2) \ge 1  - \left( \exp\{d^{\frac32h}\}\cdot (\log L(h))^2 + 1 \right) \cdot \exp\{-c_\mathbb{L} \cdot L(h)\}.$$
This shows that, if we had~$L(h) >d^{2h}$, we would get
$$\mathbb{P}\left( \CP[V \cup \mathcal{T}_h]{\tilde{G}}{\lambda}{t}(v_{L(h)-1}) = 0\; \text{for all }t \ge 0\right) \ge \mathbb{P}(F_1 \cap F_2) \xrightarrow{h \to \infty} 1.$$
On the other hand, the definition of~$L(h)$ implies that
$$\mathbb{P}\left({\xi}^{V \cup \mathcal{T}_h}_{\tilde{G},\lambda;t}(v_{L(h)-1}) > 0 \text{ for some }t \ge 0\right) \ge 1 -\cals(h)^{-1} \xrightarrow{h \to \infty} 1,$$
a contradiction.
\end{proof}

The following guarantees that if the contact process with some initial condition remains active for~$\cals(h)$ time in~$\tilde{G}$, then it is highly likely to coincide with the process started from full occupancy. This, in turn, will be applied in the proof of Lemma \ref{lem:prob_crossing} below which is an important step towards obtaining lower bounds on $L(h)$.

\begin{proposition}\label{prop:persist_couple}
If~$h$ is large enough, for any~$A \subset \tilde{V}$ we have
\begin{equation*}
\mathbb{P}\left(\CP[A]{\tilde{G}}{\lambda}{\cals(h)} \neq \varnothing,\; \CP[A]{\tilde{G}}{\lambda}{\cals(h)} \neq \CP[\tilde{V}]{\tilde{G}}{\lambda}{\cals(h)} \right) < {\cals(h)^{-2}}.
\end{equation*}
\end{proposition}
The proof of this proposition is lengthy and technical, so we postpone it to the Appendix.

We are now interested in giving an upper bound for the probability that the infection crosses~$\mathcal{L}_h$ in a single attempt. For the proof of Proposition~\ref{prop:no_go}, it will be important that this bound is given in terms of the extinction time of the infection on~$\tilde{G}$, starting from full occupancy. 

Define
$$S(h) := \mathbb{E}\left[\inf\left\{t: \CP[\tilde{V}]{\tilde{G}}{\lambda}{t} = \varnothing \right\} \right],$$
that is,~$S(h)$ is the expected amount of time it takes for the contact process on~$\tilde{G}$ with parameter~$\lambda$ started from full occupancy to die out. Also let
\begin{equation}\label{eq:def_of_p}
\mathcal{p}(\ell) = \mathcal{p}_\lambda(\ell):= \mathbb{P}\left(\barCP[\{0\}]{\mathbb{N}_0}{\lambda}(\ell) > 0 \right),
\end{equation}	
or equivalently,~$\mathcal{p}(\ell)$ is the probability that, for the contact process with parameter~$\lambda$ on a line segment of length~$\ell + 1$, an infection starting at one extremity ever reaches the other extremity. 
\begin{lemma}\label{lem:prob_crossing}
If~$h$ is large enough,
\begin{equation}\label{eq:main_bounds2}\mathcal{p}(L(h)) \le \frac{\cals(h)^3}{S(h)}.\end{equation}
\end{lemma}

\begin{proof}
Recall that~$v_0$ is the vertex of~$\mathcal{L}_h$ neighboring~$\rho$, the root of~$\mathcal{T}_h$. Let~$q(h)$ denote the probability that there is an infection path starting from~$(v_0,0)$, ending at~$(\tilde{o},t)$ for some~$t \le \cals(h)$, and entirely contained in~$\mathcal{L}_h$. Note that~$q(h) \le \mathcal{p}(L(h))$ and, by a union bound,
$$\mathcal{p}(L(h)) \leq q(h) + \mathbb{P}\left(\CP[\{v_0\}]{\mathcal{L}_h}{\lambda}{\cals(h)} \neq \varnothing \right) \stackrel{\eqref{eq:exp_decay_time}}{\le} q(h) + e^{-c_{\mathbb{Z}}\cdot \cals(h)}.$$

Next, assume that~$h$ is large enough that any vertex in~$V$ is at distance smaller than~$h$ from~$\rho$, the root of~$\mathcal{T}_h$. With this choice, we claim that for any~$A \subset \tilde{V}$,~$A \neq \varnothing$ we have
\begin{equation}\mathbb{P}\left( \CP[A]{\tilde{G}}{\lambda}{t}(\tilde{o})=1 \text{ for some } t \leq h + \cals(h)  \right) > (e^{-1}(1-e^{-{\lambda}}))^h\cdot q(h).\label{eq:aux_A_infects}\end{equation} Indeed, if~$A \cap \mathcal{L}_h \neq \varnothing$ then the left-hand side is larger than~$q(h)$ by the definition of~$q(h)$ and simple monotonicity considerations. If~$A \cap \mathcal{L}_h = \varnothing$, then by~\eqref{eq:lower_bound_reach}, with probability larger than~$\delta(h):= (e^{-1}(1-e^{-{\lambda}}))^h$,~$\rho$ gets infected within time~$h$, and conditioned on this, with probability~$q(h)$,~$\tilde{o}$ gets infected after at most additional~$\cals(h)$ units of time. Applying~\eqref{eq:aux_A_infects} and the strong Markov property repeatedly, we have
\begin{equation}
\mathbb{P}\left(\begin{array}{l}\CP[A]{\tilde{G}}{\lambda}{t} \neq \varnothing,\\[.2cm] \CP[A]{\tilde{G}}{\lambda}{r}(\tilde{o}) = 0\; \forall r \leq t \end{array}\right) \leq \left(1 - \delta(h)\cdot q(h)\right)^{\left \lfloor \frac{t}{h+\cals(h)}\right\rfloor},\; t \ge h+\cals(h).
\end{equation}

Now, letting~$S'(h):= \frac{S(h)}{4 \cals(h)}$, we have
\begin{align}\nonumber (\cals(h))^{-1} &< \mathbb{P}\left(\bar{\xi}^{V \cup \mathcal{T}_h}_{\tilde{G},\lambda}(\tilde{o}) = 0 \right) \\\nonumber&\le \mathbb{P}\left(\CP[V \cup \mathcal{T}_h]{\tilde{G}}{\lambda}{S'(h)}= \varnothing \right) + \left(1-\delta(h)\cdot q(h) \right)^{\left \lfloor \frac{S'(h)}{h+\cals(h)} \right\rfloor}\\
&\le \mathbb{P}\left( \CP[V \cup \mathcal{T}_h]{\tilde{G}}{\lambda}{S'(h)}= \varnothing \right)  + \exp\left\{-\delta(h)\cdot q(h) \cdot \left\lfloor \frac{S'(h)}{h+\cals(h)}\right \rfloor \right\}. \label{eq:want_rel}
\end{align}
where the first inequality follows from the definition of~$L(h)$, see~\eqref{eq:def_of_p_cross} and~\eqref{eq:def_of_L_cross}. We now claim that
\begin{equation}\label{eq:rel_t2}
\mathbb{P}\left( \CP[V \cup \mathcal{T}_h]{\tilde{G}}{\lambda}{S'(h)}= \varnothing \right)  < (2\cals(h))^{-1}
\end{equation}
if~$h$ is large enough. Plugging this into~\eqref{eq:want_rel}, we obtain
$$q(h) < \frac{\log(2\cals(h))\cdot (h+\cals(h))}{\delta(h)\cdot S'(h)} < \frac{4 \log(2\cals(h)) \cdot (h+\cals(h))\cdot \cals(h)}{(e^{-1}(1-e^{-\lambda}))^h \cdot S(h)} < \frac{\cals(h)^3}{S(h)}$$
for large enough~$h$, completing the proof.

It remains to prove~\eqref{eq:rel_t2}. Noting  that~$S'(h) \gg \cals(h)$ if~$h$ is large, we have
\begin{align*}&\mathbb{P}\left(\CP[V \cup \mathcal{T}_h]{\tilde{G}}{\lambda}{S'(h)}= \varnothing\right) \le \mathbb{P}\left(\CP[\tilde{V}]{\tilde{G}}{\lambda}{S'(h)}= \varnothing\right) + \mathbb{P}\left(\CP[V \cup \mathcal{T}_h]{\tilde{G}}{\lambda}{\cals(h)} \neq \CP[\tilde{V}]{\tilde{G}}{\lambda}{\cals(h)}\right).
\end{align*}
By~\autoref{lem:inv_markov}, we have
$$ \mathbb{P}\left(\CP[\tilde{V}]{\tilde{G}}{\lambda}{ S'(h)}= \varnothing\right) \leq \frac{S'(h)}{S(h)} = (4\cals(h))^{-1}.$$
Next,
\begin{align*} &\mathbb{P}\left(\CP[V \cup \mathcal{T}_h]{\tilde{G}}{\lambda}{\cals(h)} \neq \CP[\tilde{V}]{\tilde{G}}{\lambda}{\cals(h)}\right)\\
&\leq \mathbb{P}\left(\CP[V \cup \mathcal{T}_h]{\tilde{G}}{\lambda}{\cals(h)} = \varnothing \right) + \mathbb{P}\left(\CP[V \cup \mathcal{T}_h]{\tilde{G}}{\lambda}{\cals(h)} \neq \varnothing,\;\CP[V \cup \mathcal{T}_h]{\tilde{G}}{\lambda}{\cals(h)}  \neq \CP[\tilde{V}]{\tilde{G}}{\lambda}{\cals(h)}\right).
\end{align*}
Now, the first term on the right-hand side is smaller than~$\cals(h)^{-2}$ by Lemma~\ref{lem:persist0} (since~$V \cup \mathcal{T}_h \in \mathcal{A}_h$), and the second term on the right-hand side is also smaller than~$\cals(h)^{-2}$ by Proposition~\ref{prop:persist_couple}.
Putting things together gives~\eqref{eq:rel_t2} for large enough~$h$.
\end{proof}

We end this section with a lower bound on~$L(h)$, which again will be important for the proof of Proposition~\ref{prop:no_go}.
\begin{lemma}
	If~$h$ is large enough,
\begin{equation}\label{eq:lower_bound_L}
L(h) \ge d^{\frac{3h}{4}}.
\end{equation}
	\end{lemma}
	
\begin{proof} 
By the simple estimate~\eqref{eq:lower_bound_reach} and Lemma \ref{lem:prob_crossing}, we have
$$(e^{-1}(1-e^{-\lambda}))^{L(h)} \le \mathcal{p}_{{\lambda}}(L(h)) \le \frac{\cals(h)^3}{S(h)}.$$
This gives 
$$ L(h) \ge \frac{1}{\log(e(1-e^{-\lambda})^{-1})} \cdot \log\left(\frac{S(h)}{\cals(h)^3} \right).$$
Recalling that~$\cals(h) = \exp\{d^{\sqrt{h}}\}$ and noting that
$$S(h) \ge \mathbb{E}\left[\inf\{t: \CP{\mathcal{T}_h}{\lambda}{t} = \varnothing \} \right] \ge \exp\{c_\mathbb{T}\cdot  d^h \},$$
we obtain
$$L(h) \ge \frac{c_\mathbb{T} \cdot d^h - 3d^{\sqrt{h}}}{\log(e(1-e^{-\lambda})^{-1})}  > d^{\frac{3h}{4}}$$
if~$h$ is large enough.
\end{proof}

\subsection{Proof of Proposition~\ref{prop:time_spent}}
\label{ss:time_spent}

We begin with a simple consequence of Proposition~\ref{prop:persist_couple}.
\begin{lemma} If~$h$ is large enough, for any~$A \in \mathcal{A}_h$ we have
	\label{lem:time_spent0}
$$\mathbb{P}\left(\CP[A]{\tilde{G}}{\lambda}{\cals(h)}= \CP[V \cup \mathcal{T}_h]{\tilde{G}}{\lambda}{\cals(h)} \right) > 1-4\cals(h)^{-2}.$$
	\end{lemma}
\begin{proof}
Since both~$A$ and~$V \cup \mathcal{T}_h$ belong to~$\mathcal{A}_h$, Lemma~\ref{lem:persist0} gives
$$\mathbb{P}\left(\CP[A]{\tilde{G}}{\lambda}{\cals(h)} = \varnothing \right) < \frac{1}{\cals(h)^2},\quad \mathbb{P}\left(\CP[V \cup \mathcal{T}_h]{\tilde{G}}{\lambda}{\cals(h)} = \varnothing \right) < \frac{1}{\cals(h)^2}, $$
and Proposition~\ref{prop:persist_couple} gives
\begin{align*}&\mathbb{P}\left(\CP[A]{\tilde{G}}{\lambda}{\cals(h)} \neq \varnothing,\;  \CP[A]{\tilde{G}}{\lambda}{\cals(h)} \neq \CP[\tilde{V}]{\tilde{G}}{\lambda}{\cals(h)}\right) < \cals(h)^{-2},\\[.2cm]&\mathbb{P}\left(\CP[V \cup \mathcal{T}_h]{\tilde{G}}{\lambda}{\cals(h)} \neq \varnothing,\;  \CP[V \cup \mathcal{T}_h]{\tilde{G}}{\lambda}{\cals(h)} \neq \CP[\tilde{V}]{\tilde{G}}{\lambda}{\cals(h)}\right) < \cals(h)^{-2}. \end{align*}
The desired statement follows from these four inequalities.
	\end{proof}

\begin{lemma}\label{lem:time_spent_baby}
\label{lem:time_spent_baby}
If~$h$ is large enough we have, for any~$v \in V$,
\begin{equation*}
\mathbb{P}\left(\int_{\cals(h)}^\infty \CP[V \cup \mathcal{T}_h]{\tilde{G}}{\lambda}{t}(v)\;\mathrm{d}t > h \right) > 1- \frac{1}{2h}.
\end{equation*}
	\end{lemma}
\begin{proof}
	Assume~$h$ is larger than the graph diameter of~$G$, and fix~$v \in V$. We have, for any~$A \subset \tilde{V}$ with~$A \cap \mathcal{T}_h \neq \varnothing$,
	$$\mathbb{P}\left(\int_0^{4h}\CP[A]{\tilde{G}}{\lambda}{t}(v)\;\mathrm{d}t > h\right) \ge e^{-2h}\cdot (e^{-1}\cdot (1-e^{-\lambda}))^{2h}.$$	
	Indeed, by the estimate~\eqref{eq:lower_bound_reach} we have that, with probability at least~$(e^{-1}\cdot (1-e^{-\lambda}))^{2h}$,~$v$ becomes infected  before time~$2h$, and then it remains infected for time~$2h$ (by having no recovery marks) with probability~$e^{-2h}$. By iterating this, we obtain
	\begin{align*}&\mathbb{P}\left(\CP[\mathcal{T}_h]{\mathcal{T}_h}{\lambda}{2\cals(h)} \neq \varnothing,\; \int_{\cals(h)}^{2\cals(h)}\CP[V \cup \mathcal{T}_h]{\tilde{G}}{\lambda}{t}(v) \;\mathrm{d}t \le h\right) \\&\hspace{4cm} < \left( 1- e^{-2h}\cdot (e^{-1}\cdot (1-e^{-\lambda}))^{2h} \right)^{\lfloor \cals(h)/(4h)\rfloor} \ll \frac{1}{4h}.\end{align*}
We therefore have
	\begin{align*}
	&\mathbb{P}\left(\int_{\cals(h)}^{2\cals(h)}\CP[V \cup \mathcal{T}_h]{\tilde{G}}{\lambda}{t}(v)\;\mathrm{d}t \le h \right)\\
	 &\le \mathbb{P}\left(\CP[\mathcal{T}_h]{\mathcal{T}_h}{\lambda}{2\cals(h)} = \varnothing \right) + \mathbb{P}\left(\CP[\mathcal{T}_h]{\mathcal{T}_h}{\lambda}{2\cals(h)} \neq \varnothing,\;\int_{\cals(h)}^{2\cals(h)}\CP[V \cup \mathcal{T}_h]{\tilde{G}}{\lambda}{t}(v)\;\mathrm{d}t \le h \right) \\[.2cm]
	&\quad\stackrel{\eqref{eq:ud_markov},\eqref{eq:metastability}}\le \frac{2\cals(h)}{\exp\{c_\mathbb{T}\cdot d^h\}} + \left( 1- e^{-2h}\cdot (e^{-1}\cdot (1-e^{-\lambda}))^{2h} \right)^{\lfloor \cals(h)/(4h)\rfloor} \ll \frac{1}{2h}
	\end{align*}
	if~$h$ is large, which implies the statement of the lemma.
\end{proof}

\begin{lemma} \label{lem:time_spent} If~$h$ is large enough we have
\begin{equation}
\mathbb{P}\left( \int_{\cals(h)}^\infty\CP[V \cup \mathcal{T}_h]{\tilde{G}}{\lambda}{t}(\tilde{o})\;\mathrm{d}t > h \right) > 1 - \frac{1}{2h}.
\end{equation}
\end{lemma}
\begin{proof}
We will separately prove that
\begin{equation}\label{eq:time_spent_aux0}\mathbb{P}\left(\int_{0}^\infty \CP[V \cup \mathcal{T}_h]{\tilde{G}}{\lambda}{t}(\tilde{o})\;\mathrm{d}t  > h  \right) > 1 - \frac{1}{4h}\end{equation}	
and
\begin{equation}\label{eq:time_spent_aux}\mathbb{P}\left(\int_{0}^{\cals(h)} \CP[V \cup \mathcal{T}_h]{\tilde{G}}{\lambda}{t}(\tilde{o})\;\mathrm{d}t = 0 \right) > 1 - \frac{1}{4h};\end{equation}
the desired result will then follow.

For~\eqref{eq:time_spent_aux0}, let~$u_1 := v_{L(h)-2},u_2:= v_{L(h)-1}$ be such that~$u_1,u_2,\tilde{o}$ (in this order) are the three last vertices in~$\mathcal{L}_h$, as we move away from~$\mathcal{T}_h$. By Lemma~\ref{lem:well_def} and the definition of~$L(h)$ we have
\begin{equation}\label{eq:ui}\mathbb{P}\left(\barCP[V \cup \mathcal{T}_h]{\tilde{G}}{\lambda}(u_i) = 0 \right)  <\cals(h)^{-1},\quad i = 1,2.\end{equation}
Let~$G'$ denote~$\tilde{G}$ after removing~$u_2$ and~$\tilde{o}$. Define the random set of times
$$I:= \left\{t \ge 0: u_1 \in \CP[V \cup \mathcal{T}_h]{G'}{\lambda}{t}\right\}.$$
We have
\begin{equation}\label{eq:expo_inf}\mathbb{P}\left(\barCP[V \cup \mathcal{T}_h]{\tilde{G}}{\lambda}(u_2) = 0\right) = \mathbb{E}\big[1- e^{-{\lambda} |I|} \big],\end{equation}
where~$|I|$ denotes the Lebesgue measure of~$I$. To justify this, note that the first time that~$u_2$ becomes infected in~$(\xi^{V \cup \mathcal{T}_h}_{\tilde{G},\lambda;t})_{t \ge 0}$ is necessarily through a transmission from~$u_1$. Hence, one can decide if~$u_2$ is ever infected in this process by inspecting whether there is a point in time at which~(1)~$u_1$ is infected in process confined to~$G'$, and~(2) there is a transmission arrow from~$u_1$ to~$u_2$.  The number of such time instants is a Poisson random variable with parameter~${\lambda} |I|$, justifying~\eqref{eq:expo_inf}. 

We bound 
\begin{align*}
\cals(h)^{-1} \stackrel{\eqref{eq:ui}}{>} \mathbb{P}\left(\barCP[V \cup \mathcal{T}_h]{\tilde{G}}{\lambda}(u_2)= 0  \right) \stackrel{\eqref{eq:expo_inf}}{=} \mathbb{E}\big[e^{-{\lambda}|I|}\big] \ge e^{-{\lambda} h^2}\cdot \mathbb{P}(|I| < h^2),
\end{align*}
so
\begin{equation}\label{eq:less_than_h2}
\mathbb{P}(|I|<h^2) \le e^{{\lambda} h^2}\cdot \cals(h)^{-1} \ll \frac{1}{8h}
\end{equation}
for~$h$ large enough.

We next claim that
\begin{equation}
\label{eq:claim_inf} \mathbb{P}\left(\left. \barCP[V \cup \mathcal{T}_h]{\tilde{G}}{\lambda}(\tilde{o}) \le h \right| |I| \ge h^2 \right) < e^{-h}.
\end{equation}
To prove this, we observe that on the event~$\{|I| \ge h^2\}$, we can find an increasing sequence of times~$S_0,\ldots, S_{\lfloor h^2/2 \rfloor} \in I$ with
$$|I \cap [S_j+1,S_{j+1}]| \ge 2\quad \text{for each }j.$$
Next, note that for each interval~$[S_j,S_{j+1}]$, with a probability that is positive and depends only on~${\lambda}$, the infection is sent to~$\tilde{o}$ and remains there for one unit of time. This occurring independently in different time intervals,~\eqref{eq:claim_inf} follows from a simple Chernoff bound. Now,~\eqref{eq:time_spent_aux0} follows from~\eqref{eq:less_than_h2} and~\eqref{eq:claim_inf}.

We now turn to~\eqref{eq:time_spent_aux}. Note that the event inside the probability there is contained in the event that there is an infection path starting  at some time~$s$ and ending at some time~$t$ with~$s \le t \le \cals(h)$, connecting the two endpoints of~$\mathcal{L}_h$. By~\autoref{cor:exp_decay_space}, the probability that such a path exists is smaller than
$$(\cals(h)+1) \cdot \exp\{-c_\L \cdot L(h)\} \ll \frac{1}{4h}$$
if~$h$ is large enough.
\end{proof}

\begin{proof}[Proof of Proposition~\ref{prop:time_spent}]
The statements follow readily from Lemmas~\ref{lem:time_spent0},~\ref{lem:time_spent_baby} and~\ref{lem:time_spent}.
\end{proof}

\subsection{Proof of Proposition~\ref{prop:no_go}}
Proving Proposition~\ref{prop:no_go} is now just a matter of putting together bounds that were obtained  earlier.
\label{ss:lower_lambda}
\begin{proof}[Proof of Proposition~\ref{prop:no_go}]
Fix $\lambda' < \lambda$. Let~$B$ be the event that, in the graphical construction with parameter~$\lambda'$, there is an infection path starting from~$(v_0,s)$ for some~$s \le 2 \cals(h) \cdot S(h)$ (where~$v_0$ is the vertex of~$\mathcal{L}_h$ neighboring the root~$\rho$ of~$\mathcal{T}_h$), ending at~$(\tilde{o},t)$ for some~$t > s$, and entirely contained in~$\mathcal{L}_{h}$. Then, by a union bound,
	\begin{equation}\label{eq:two_parts} \mathbb{P}\left(\barCP[A]{\tilde{G}}{\lambda'}(\tilde{o}) > 0 \right) \leq \mathbb{P}\left(\CP[\tilde{V}]{\tilde{G}}{\lambda'}{2\cals(h) \cdot S(h)} \neq \varnothing\right) + \mathbb{P}_{\lambda'}(B).\end{equation}
	The first term is bounded using Markov's inequality and monotonicity:
	\begin{align*}\mathbb{P}\left(\CP[\tilde{V}]{\tilde{G}}{\lambda'}{2\cals(h) S(h)} \neq \varnothing\right) &\leq  (2\cals(h))^{-1}.
	\end{align*}
Next, note that the occurrence of~$B$ depends only on the graphical construction of the contact process with parameter~$\lambda'$ on the line segment connecting~$v_0$ and~$\tilde{o} = v_{L(h)}$. Therefore, using Lemma~\ref{lem:exp_decay} (and also recalling the definition of~$\mathcal{p}$ from~\eqref{eq:def_of_p}), we have
\begin{equation}
\mathbb{P}_{\lambda'}(B) \le (2\cals(h)\cdot S(h)+1)\cdot e\cdot \mathcal{p}_{\lambda'}(L(h)) \le 7\cals(h)\cdot S(h) \cdot \mathcal{p}_{\lambda'}(L(h)).
\end{equation}
Bounding the right-hand side using Lemma~\ref{lem:line_lambda}, we obtain
$$\mathbb{P}_{\lambda'}(B) \le 7\cals(h)\cdot S(h)\cdot \mathcal{p}_\lambda(L(h))\cdot (\eta_{\lambda, \lambda'})^{-L(h)}.$$
By using~$\mathcal{p}(L(h)) \le \cals(h)^3/S(h)$ as in~\eqref{eq:main_bounds2} and~$L(h) \ge d^{\frac{3h}{4}}$ as in~\eqref{eq:lower_bound_L}, the right-hand side above is smaller than 
$$7\cals(h)^4 \cdot (\eta_{\lambda,\lambda'})^{-d^{\frac{3h}{4}}}, $$
which is much smaller than~$(2\cals(h))^{-1}$ if~$h$ is large enough (depending on~$\lambda$ and~$\lambda'$).
\end{proof}

\newpage
\section*{Appendix: proof of Proposition~\ref{prop:persist_couple} }
\label{app1}
	We will first state and prove some auxiliary claims.
\begin{claim}
		For any~$A \subset \tilde{V}\backslash \mathcal{T}_h$ we have
		\begin{equation*}
		\P\left(\text{either } \CP[A]{\tilde{G}}{\lambda}{\sqrt{\cals(h)}} = \varnothing \;\text{ or }\; \CP[A]{\tilde{G}}{\lambda}{t} \cap \mathcal{T}_h \neq \varnothing \text{ for some }t \le \sqrt{\cals(h)}\right) \ge \frac12.
		\end{equation*}
\end{claim}
	\begin{proof} Let~$G'$ be the graph obtained  by removing~$\mathcal{T}_h$ from~$\tilde{G}$ (so that~$G'$ is the disconnected union of~$G$ and~$\mathcal{L}_h$). 
		The complement of the event in the probability above is
		$$\left\{\CP[A]{G'}{\lambda}{\sqrt{\cals(h)}} \neq \varnothing \right\} \subset \left\{\CP[V \cup \mathcal{L}_h]{G'}{\lambda}{\sqrt{\cals(h)}} \neq \varnothing \right\} \subset \{\tau_1 > \sqrt{\cals(h)}\} \cup \{\tau_2 > \sqrt{\cals(h)}\},$$
where
$$\tau_1 = \inf\left\{t:\CP[V]{G}{\lambda}{t} = \varnothing \right\},\quad \tau_2 = \inf\left\{t: \CP{\mathcal{L}_h}{\lambda}{t} = \varnothing\right\}.$$
Since~$G$ is fixed  while~$h$ can be  taken arbitrarily large, we can assume
$$\mathbb{P}\left(\tau_1 > \sqrt{\cals(h)}\right) < \frac14.$$
Next, noting that, by \autoref{prop:bounds},
$$\cals(h) = \exp\{d^{\sqrt{h}}\} \gg \left(\log \left( d^{\frac{3h}{2}}\right)\right)^2  \stackrel{\eqref{eq:main_bounds}}{\geq} (\log L(h))^2,$$
we have, by \autoref{lem:more_than_log},
$$\mathbb{P}\left(\tau_2 > \sqrt{\cals(h)}\right) \le \mathbb{P}\left(\tau_2 > (\log L(h))^2 \right) \stackrel{\eqref{eq:more_than_log}}{<} \frac14$$
if~$h$ is large enough.	\end{proof}
For the next two claims we let $c_\T$ be as in \autoref{basic_trees}.
	\begin{claim}\label{cl:second}
		For any~$A \subset \tilde{V}$ we have
		$$\mathbb{P}\left( \text{either }\;\CP[A]{\tilde{G}}{\lambda}{2\sqrt{\cals(h)} } = \varnothing \;\text{ or }\;\CP[A]{\tilde{G}}{\lambda}{2\sqrt{\cals(h)}} \supset \CP{\mathcal{T}_h}{\lambda}{2\sqrt{\cals(h)}} \right) > \frac{c_\mathbb{T}}{2}.$$
	\end{claim}
	\begin{proof}
		Define
		$$\tau':= \inf\left\{t: \CP[A]{\tilde{G}}{\lambda}{t} = \varnothing\right\},\quad \tau'':= \inf\left\{t: \CP[A]{\tilde{G}}{\lambda}{t}\cap \mathcal{T}_h \neq \varnothing \right\}$$
		and let~$\tau = \min(\tau',\tau'')$.
		By the first claim we have
		\begin{equation}\mathbb{P}(\tau \le \sqrt{\cals(h)}) \ge \frac12.\label{eq:aux_2th0}\end{equation} Next, note that
		\begin{equation}\mathbb{P}\left(\CP[A]{\tilde{G}}{\lambda}{2\sqrt{\cals(h)}} = \varnothing \mid \tau = \tau' \le \sqrt{\cals{(h)}}\right) = 1.\label{eq:aux_2th1}\end{equation}
		We will prove that
		\begin{equation}\mathbb{P}\big(\CP[A]{\tilde{G}}{\lambda}{2\sqrt{\cals(h)}} \supset \CP{\mathcal{T}_h}{\lambda}{2\sqrt{\cals(h)}} \mid \tau = \tau'' \le \sqrt{\cals(h)}\big) > c_\mathbb{T}.\label{eq:aux_2th2}\end{equation}
		Taken together,~\eqref{eq:aux_2th0}, \eqref{eq:aux_2th1} and~\eqref{eq:aux_2th2} give the statement of the claim.
		
		To prove~\eqref{eq:aux_2th2}, we first introduce some notation. Given~$A'\subset \mathcal{T}_h$, we write
		$$\xi^{A'}_{\mathcal{T}_h,\lambda; t_1,t_2}(x) := \mathds{1}\{A' \times \{t_1\} \rightsquigarrow (x,t_2)\},\quad t_1 \le t_2,\; x \in \mathcal{T}_h. $$
		Note that~$(\xi^{A'}_{\mathcal{T}_h,\lambda; t_1,t_1+s}:s\ge 0)$ has same distribution as~$(\xi^{A'}_{\mathcal{T}_h,\lambda;s}:s \ge 0)$. Next, on the event~$\{\tau''<\infty\}$ let~$A':=\xi^A_{G,\lambda;\tau''} \cap \mathcal{T}_h$. Define the event
		$$B:= \{\tau'' < \infty \}\cap \left\{ \CP[A']{\mathcal{T}_h}{\lambda}{\tau'', \tau'' + \sqrt{\cals(h)}} \supset \CP{\mathcal{T}_h}{\lambda}{\tau'', \tau''+\sqrt{\cals(h)}} \right\}. $$
		By Proposition~\ref{basic_trees} and the strong Markov property we have~$\mathbb{P}(B\mid \tau''< \sqrt{\cals(h)}) > c_\mathbb{T}$. Moreover, on~$B$ we have
		\begin{align*}\CP[A]{\tilde{G}}{\lambda}{2\sqrt{\cals(h)}} &\supset \CP[A']{\mathcal{T}_h}{\lambda}{\tau'', 2\sqrt{\cals(h)} }\supset \CP{\mathcal{T}_h}{\lambda}{\tau'',2\sqrt{\cals(h)}} \supset \CP{\mathcal{T}_h}{\lambda}{2\sqrt{\cals(h)}}.
		\end{align*}
		This completes the proof.
	\end{proof}
	\begin{claim}\label{cl:final} For any~$A \subset \tilde{V}$ and~$h$ large enough we have
		\begin{equation}\mathbb{P}\left(\begin{array}{l}\text{either }\CP[A]{\tilde{G}}{\lambda}{\cals(h)/2} = \varnothing \text{ or } \\[.2cm]\CP[A]{\tilde{G}}{\lambda}{\cals(h)/2} \supset \CP{\mathcal{T}_h}{\lambda}{\cals(h)/2}\end{array} \right) > 1-\left(1- \frac{c_\mathbb{T}}{2} \right)^{\left \lfloor \sqrt{\cals(h)}/4\right\rfloor}.\label{eq:from_claim}\end{equation}
	\end{claim}
In words, the event in the probability in the left-hand side can be described by saying that  one of two alternatives has to hold true for the contact process on~$\tilde{G}$ started from~$A$. The first alternative is that by time~$\cals(h)/2$, this process dies. The second alternative is that at time~$\cals(h)/2$, the occupation of this process inside~$\mathcal{T}_h$ is large enough that it contains the set~$\xi^{\mathcal{T}_h}_{\mathcal{T}_h,\lambda;\cals(h)/2}$. This set is obtained by running the contact process only inside~$\mathcal{T}_h$, from time 0 to time~$\cals(h)/2$, starting from full occupancy at time~0.

	\begin{proof}
		For~$1 \leq i \leq \left\lfloor \sqrt{\cals(h)}/4\right\rfloor$, define the event
		$$F_i := \left\{ \CP[A]{\tilde{G}}{\lambda}{i\cdot 2\sqrt{\cals(h)}} = \varnothing \right\} \cup \left\{ \CP[A]{\tilde{G}}{\lambda}{i\cdot 2\sqrt{\cals(h)}} \supset \CP{\mathcal{T}_h}{\lambda}{i \cdot 2\sqrt{\cals(h)}}  \right\}.$$
		We then note that the event in the probability in~\eqref{eq:from_claim} is contained in~$\cup F_i$, and by Claim~\ref{cl:second},
		$$\mathbb{P}\left(\cap_i F_i^c \right) \le \left(1- \frac{c_\mathbb{T}}{2} \right)^{\left \lfloor \sqrt{\cals(h)}/4\right\rfloor}.$$
	\end{proof}
	
	We now introduce notation for the time dual of the contact process: if~$G' = (V',E')$ is a graph, we write
	$$\tilde{\xi}^A_{{G}',\lambda;s,t}(x):= \mathds{1}\{(x,s) \rightsquigarrow A \times \{t\} \},\quad x \in V',\;A \subset V',\; s \le t$$
	(as usually, we abuse notation and sometimes treat~$\tilde{\xi}^A_{G',\lambda; s,t}$ as a subset of~$V'$ rather than a configuration of~0's and~1's). Note that for any $s,t > 0$ and $x,y \in \tilde{V}$ we have
	\begin{equation}\label{eq:dual_inclusion}
	\left\{ \CP[\{x\}]{G'}{\lambda}{s}\cap \tilde{\xi}^{\{y\}}_{{G}',\lambda;s,t}  \neq \varnothing \right\} \subset \left\{ y \in \CP[\{x\}]{G'}{\lambda}{t} \right\},\qquad s\le t.
	\end{equation}
	
\begin{proof}[Proof of Proposition~\ref{prop:persist_couple}]
	Fix $x \in \tilde{V}$ and define
	$$E_1 := \left\{\text{for all }x \in \tilde{V}, \text{ either }\CP[\{x\}]{\tilde{G}}{\lambda}{\cals(h)/2} = \varnothing \text{ or } \CP[\{x\}]{\tilde{G}}{\lambda}{\cals(h)/2} \supset \CP{\mathcal{T}_h}{\lambda}{\cals(h)/2} \right\}$$

Note that~$E_1$ is the joint occurrence of the event of Claim~\ref{cl:final} (that is, the event inside the probability in the left-hand side of~\eqref{eq:from_claim}), with~$A$ ranging over all sets of the form~$\{x\}$, with~$x \in \tilde{V}$.

%In other words, $E_1$ is the event where $(x,0) \rightsquigarrow (y,\cals(h)/2)$ for every pair $(x,y) \in \tilde{V} \times \mathcal{T}_h$ such that $(x,0) \rightsquigarrow \tilde{V}\times \{\cals(h)/2\}$ and $\mathcal{T}_h \times \{0\} \rightsquigarrow (y,\cals(h)/2) \}$ by a path fully contained in $\mathcal{T}_h$.

By Claim \ref{cl:final}, $E_1$ has probability larger than~$1 - |\tilde{V}|\cdot \left(1- \frac{c_\mathbb{T}}{2} \right)^{\left \lfloor \sqrt{\cals(h)}/4\right\rfloor}$. Note also that for any $A \subset \tilde{V}$ we have
	\begin{equation}\label{eq:proof_prop5}
	E_1 \cap \left\{ \CP[A]{\tilde{G}}{\lambda}{\cals(h)/2} \neq \varnothing \right\} \subset \left\{ \CP[A]{\tilde{G}}{\lambda}{\cals(h)/2} \supset \CP{\mathcal{T}_h}{\lambda}{\cals(h)/2} \right\}.
	\end{equation}
%	has probability larger than~$1 - |\tilde{V}|\cdot \left(1- \frac{c_\mathbb{T}}{2} \right)^{\left \lfloor \sqrt{\cals(h)}/4\right\rfloor}$.

Define the event
	$$E_2 := \left\{\begin{array}{l}\text{for all }x \in \tilde{V}, \text{ either } \tilde{\xi}^{\{x\}}_{\tilde{G},\lambda;\cals(h)/2,\cals(h)} = \varnothing\\[.2cm] \text{or } \tilde{\xi}^{\{x\}}_{\tilde{G},\lambda;\cals(h)/2,\cals(h)} \supset \tilde{\xi}^{\mathcal{T}_h}_{\mathcal{T}_h,\lambda;\cals(h)/2,\cals(h)} \end{array}  \right\}.$$
This event can be interpreted in the same way as~$E_1$, except that it pertains to the  dual process. 	By invariance of Poisson processes under time reversal,~$E_2$ has the same probability as~$E_1$. It is also the case that, for any~$A \subset \tilde{V}$,
	\begin{equation}\label{eq:proof_prop5_dual}
	E_2 \cap \left\{ \tilde{\xi}^{A}_{\tilde{G},\lambda;\cals(h)/2,\cals(h)} \neq \varnothing \right\} \subset \left\{ \tilde{\xi}^{A}_{\tilde{G},\lambda;\cals(h)/2,\cals(h)} \supset \tilde{\xi}^{\mathcal{T}_h}_{\mathcal{T}_h,\lambda;\cals(h)/2,\cals(h)} \right\}.
	\end{equation}
	Finally, Proposition~\ref{basic_trees} implies that if~$h$ is large enough, the event
	$$E_3 := \{\CP{\mathcal{T}_h}{\lambda}{\cals(h)} \neq \varnothing\} = \left\{\CP{\mathcal{T}_h}{\lambda}{\cals(h)/2} \cap \tilde{\xi}^{\mathcal{T}_h}_{\mathcal{T}_h,\lambda;\cals(h)/2,\cals(h)}\neq \varnothing \right\}$$
	has probability larger than~$1-{\exp\{-c_\mathbb{T} d^h\}}$. Putting our bounds together, we have
	$$\mathbb{P}(E_1^c \cup E_2^c \cup E_3^c) \leq 2|\tilde{V}|\cdot \left(1- \frac{c_\mathbb{T}}{2} \right)^{\left \lfloor \sqrt{\cals(h)}/4\right\rfloor} + \exp\{-c_\mathbb{T} d^h\} \ll  \cals(h)^{-2}$$
	if~$h$ is large enough.
	
	We now claim that for any~$A \subset \tilde{V}$ we have
	
	\begin{equation*}
	E_1\cap E_2 \cap E_3 \subset \left\{  \text{either }\CP[A]{\tilde{G}}{\lambda}{\cals(h)} = \varnothing  \;\text{ or }\; \CP[A]{\tilde{G}}{\lambda}{\cals(h)} = \CP[\tilde{V}]{\tilde{G}}{\lambda}{\cals(h)} \right\}.
	\end{equation*}
To prove this, it suffices to prove
\begin{equation*}\label{eq:contain_coup}
E_1 \cap E_2 \cap E_3 \cap \left\{\CP[A]{\tilde{G}}{\lambda}{\cals(h)} \neq \varnothing\right\} \subset \left\{\CP[A]{\tilde{G}}{\lambda}{\cals(h)} = \CP[\tilde{V}]{\tilde{G}}{\lambda}{\cals(h)}\right\} \quad \forall A \subset \tilde{V},
\end{equation*}
which in turn is implied by proving:
\begin{equation}\label{eq:contain_coup}
E_1 \cap E_2 \cap E_3 \cap \left\{\CP[A]{\tilde{G}}{\lambda}{\cals(h)/2} \neq \varnothing\right\} \subset \left\{\CP[A]{\tilde{G}}{\lambda}{\cals(h)} = \CP[\tilde{V}]{\tilde{G}}{\lambda}{\cals(h)}\right\} \quad \forall A \subset \tilde{V}.
\end{equation}
So for the rest of this proof, we assume that the event in the left-hand side of~\eqref{eq:contain_coup} occurs. Fix~$x \in \tilde{V}$; we would like to prove that~$\CP[A]{\tilde{G}}{\lambda}{\cals(h)}(x) = \CP[\tilde{V}]{\tilde{G}}{\lambda}{\cals(h)}(x)$. This is immediate in case~$\CP[\tilde{V}]{\tilde{G}}{\lambda}{\cals(h)}(x) = 0$, so from now on we also assume that~$\CP[\tilde{V}]{\tilde{G}}{\lambda}{\cals(h)}(x) =1$, which also implies that
\begin{equation}
\tilde{\xi}^{\{x\}}_{\tilde{G},\lambda;\cals(h)/2,\cals(h)} \neq \varnothing.
\end{equation}

Now, since~$E_1 \cap \left\{\CP[A]{\tilde{G}}{\lambda}{\cals(h)/2} \neq \varnothing\right\}$ occurs, by~\eqref{eq:proof_prop5} we have that
\begin{equation*}
\CP[A]{\tilde{G}}{\lambda}{\cals(h)/2} \supset \CP[\mathcal{T}_h]{\mathcal{T}_h}{\lambda}{\cals(h)/2}.
\end{equation*}
Moreover, since~$E_2 \cap \left\{\tilde{\xi}^{\{x\}}_{\tilde{G},\lambda;\cals(h)/2,\cals(h)}  \neq \varnothing\right\}$ occurs, by~\eqref{eq:proof_prop5_dual} we have that 
\begin{equation*}
\tilde{\xi}^{\{x\}}_{\tilde{G},\lambda;\cals(h)/2,\cals(h)}  \supset \tilde{\xi}^{\mathcal{T}_h}_{\mathcal{T}_h,\lambda;\cals(h)/2,\cals(h)}. 
\end{equation*}
Finally, since~$E_3$ occurs, there exists some~$z \in \CP[\mathcal{T}_h]{\mathcal{T}_h}{\lambda}{\cals(h)/2} \cap \tilde{\xi}^{\mathcal{T}_h}_{\mathcal{T}_h,\lambda;\cals(h)/2,\cals(h)}$, so the two set inclusions we just obtained imply that~$z$ is both in~$\CP[A]{\tilde{G}}{\lambda}{\cals(h)/2}$ and in~$\tilde{\xi}^{\{x\}}_{\tilde{G},\lambda;\cals(h)/2,\cals(h)}$. We thus have
$$A \times \{0\} \rightsquigarrow (z,\cals(h)/2) \rightsquigarrow (x,\cals(h)),$$
so~$\CP[A]{\tilde{G}}{\lambda}{\cals(h)}(x) = 1$ follows as desired.

\end{proof}

\end{document}